\documentclass{siamltex}
\usepackage{amsmath,amssymb}
\usepackage{latexsym}
\usepackage{algorithm}
\usepackage{algorithmic}
\usepackage{subfig}
\usepackage{graphicx}

\newtheorem{assumption}{Assumption}
\newtheorem{Exa}{Example}[section]

\newcommand{\ri}{\makebox{\rm ri}}
\newcommand{\parsub}{\makebox{\rm par}}
\newcommand{\mumin}{\mu_{\mbox{\scriptsize\rm min}}}
\newcommand{\mumax}{\mu_{\mbox{\scriptsize\rm max}}}
\newcommand{\xnew}{x^+}
\newcommand{\xsupnew}{x^{\mbox{\scriptsize\rm new}}}

\newcommand{\ip}[2]{\mbox{$\langle #1,#2 \rangle$}}
\def\tto{\;{\lower 1pt \hbox{$\rightarrow$}}\kern -10pt
           \hbox{\raise 2.8pt \hbox{$\rightarrow$}}\;}

\newcommand{\beq}{\begin{equation}}
\newcommand{\eeq}{\end{equation}}

\newcommand{\R}{\Re}

\newcommand{\cA}{{\cal A}}

\newcommand{\cC}{{\cal C}}

\newcommand{\cM}{{\cal M}}
\newcommand{\cN}{{\cal N}}

\newcommand{\cU}{{\cal U}}
\newcommand{\cV}{{\cal V}}

\def\eqnok#1{(\ref{#1})}

\newcommand{\dist}{\mbox{\rm dist}\,}

\def\sjwcomment#1{{\em [SJW: #1]}}

\def\noprint#1{}

\newcommand{\btab}{\begin{tabbing}
\ \ \= thenn \= thenn \= thenn \= thenn \= thenn \= \kill}
\newcommand{\etab}{\end{tabbing}}

\def\half{{\textstyle{1\over 2}}}

\title{A proximal method for composite minimization\footnote{\today}}

\author{A.~S. Lewis and S.~J. Wright}

\begin{document}

\maketitle

\begin{abstract}
  We consider minimization of functions that are compositions of convex or
  prox-regular functions (possibly extended-valued) with smooth vector functions. A wide variety
  of important optimization problems fall into this framework.
  We describe an algorithmic framework based on a subproblem constructed from a linearized approximation
  to the objective and a regularization term.   Properties of local solutions of this subproblem underlie both a global convergence result and an identification property of the active manifold containing the solution of the original problem.  Preliminary computational results on both convex and nonconvex examples are promising.
\end{abstract}



\section{Introduction}
\label{sec:statement}

We consider optimization problems of the form
\begin{equation} \label{hc}
\min_x\, h\big(c(x)\big),
\end{equation}
where the inner function $c:\R^n \to \R^m$ is smooth.  
The outer function $h: \R^m \to [-\infty,+\infty]$ may be
nonsmooth, but is usually convex (even polyhedral), and sufficiently
well-structured to allow us to solve, relatively easily, subproblems
of the form
\begin{equation} \label{gen-prox}
\min_d \, h\big(\Phi(d)\big) + \frac{\mu}{2} |d|^2,
\end{equation}
for {\em affine} maps $\Phi$ and scalars $\mu>0$ (where $|\cdot|$
denotes the Euclidean norm).  We analyze a ``proximal'' method for the
problem \eqnok{hc}.  In its simplest form, for a finite convex
function $h$, the method is shown below as
Algorithm~\ref{alg:proxdescent1}.
%%  We fix constants $\tau>1$, $\sigma \in (0,1)$, and $\mumin > 0$,
%% and initialize by choosing a point $x \in \R^n$ and a constant $\mu
%% \ge \mumin$.

\begin{algorithm}[ht!]
\caption{ProxDescent for Finite Convex $h$}
\label{alg:proxdescent1}
\begin{algorithmic}
\STATE  Define constants $\tau>1$, $\sigma \in (0,1)$,  and $\mumin > 0$; 
\STATE Choose $x_0 \in \R^n$, $\mu_0 \ge \mumin$; 
\FOR{$k=0,1,2,\dotsc$}
\STATE Set accept $\leftarrow$ false; 
\WHILE{not accept}
\STATE Find the minimizer $d$ of the function $h\big(c(x) + \nabla c(x)d\big) + \half\mu|d|^2$
\STATE (terminating if $d=0$);
\IF{$h\big(c(x)\big) - h\big(c(x+d)\big) \ge \sigma  
\left[ h\big(c(x)\big) - h(c(x)+\nabla c(x) d) \right]$}
\STATE $\mu \leftarrow \max(\mumin, \mu/\tau)$; 
\STATE accept $\leftarrow$ true; \\
\ELSE
\STATE $\mu  \leftarrow \tau \mu$;
\ENDIF
\ENDWHILE
\STATE $x \leftarrow x+d$;
\ENDFOR
\end{algorithmic}
\end{algorithm}

%% \noindent
%% (The method terminates if the subproblem is solved by $d=0$.)

The method repeatedly solves a {\em proximal linearized subproblem} of
the form
\begin{equation} \label{pls} 
% \mbox{PLS}(x,\mu): \;\;\; 
\min_{d} \, h_{x,\mu}(d) := 
h\big(c(x)+ \nabla c(x) d\big) + \frac{\mu}{2} |d|^2, 
\end{equation}
to find a trial step $d$, where the linear map $\nabla c(x):\R^n \to
\R^m$ is the derivative of the map $c$ at $x$ (representable by the $m
\times n$ Jacobian matrix).  In the algorithmic framework that we
discuss later, where the function $h$ is not restricted to being
polyhedral or convex, the subproblem solution $d$ is just a first
approximation to the step.  If $h$ is sufficiently well-structured ---
an assumption we make concrete using ``partial smoothness,'' a
generalization of the idea of an active set in nonlinear programming
--- we may then be able to enhance the step, possibly with the use of
higher-order derivative information.

Although many important problems of the form \eqnok{hc} involve finite
convex functions $h$, we explore extensions to broader classes of
functions $h$.  Specifically, we allow that
\begin{itemize}
\item
$h$ may be extended-valued, allowing it to incorporate constraints
  that must be enforced;
\item
$h$ is ``prox-regular'' rather than convex.
\end{itemize}
(We note in passing that our analysis extends easily to the case where
the function $c$ is defined only locally.)  This broader framework
requires additional technical overhead, but we point out throughout
the simplifications that are available in the case of continuous
convex $h$, and in particular polyhedral $h$.

\subsection{Outline}

In the next subsection, we discuss the building blocks from
variational analysis that are used in later sections, focusing on key
ideas that may be unfamiliar to many readers --- ``prox-regularity''
and ``partial smoothness'' --- and deferring more standard formal
definitions to an Appendix.  In Section~\ref{sec:examples}, we
describe how a wide variety of examples can be posed as composite
optimization problems of the form \eqref{hc}.  We include problems
from approximation, nonlinear programming, and regularized
minimization, including nonconvex examples.  Given its prevalence,
historical importance, and significance in building intuition and
terminology, we pay particular attention to the case in which the
function $h$ is finite and polyhedral.  Next, in
Section~\ref{sec:related}, we survey the extensive body of related
work.

Section~\ref{sec:sized} contains our main theoretical tools,
pertaining to the subproblem \eqref{pls}.  We note that, since any local
solution $\bar x$ for the problem \eqnok{hc} is {\em critical\/}
(meaning that $0 \in \partial (h \circ c)(\bar{x})$, where $\partial$
denotes the subdifferential), a chain rule typically implies the
existence of a vector $\bar{v}$ such that
\begin{equation} \label{hc_necc}
\bar{v} \in \partial h(\bar c) \cap \mbox{\rm Null}(\nabla c(\bar x)^*),
\end{equation}
where $\bar{c} := c(\bar{x})$ and $\mbox{}^*$ denotes the adjoint map.
In examples, we can interpret the vector $\bar{v}$ as a Lagrange
multiplier.  We begin by showing that, when the current point $x$ is
near the critical point $\bar{x}$, the proximal linearized subproblem
\eqref{pls} has a local solution $d$ of size $O(| x-\bar{x} |)$.  To
illuminate this idea, consider first a function $h$ that is convex,
lower semicontinuous, and never $-\infty$.  Assuming that the vector
$c(x)+ \nabla c(x) d$ lies in the domain of $h$ for some step $d \in
\R^n$, the subproblem (\ref{pls})
% PLS$(x,\mu)$
involves minimizing a strictly convex function with nonempty compact
level sets, and thus has a unique solution $d=d(x)$.  If we assume
slightly more --- that $c(x) + \nabla c(x)d$ lies in the relative
interior of the domain of $h$ for some $d$ (as holds obviously if $h$
is continuous at $c(x)$), a standard chain rule from convex analysis
implies that $d=d(x)$ is the unique solution of the following
inclusion:
\begin{equation} \label{pls_opt}
\nabla c(x)^* v + \mu d = 0, \;\; 
\makebox{\rm for some $v \in \partial h\big(c(x) + \nabla c(x) d\big)$}.
\end{equation}

When $h$ is prox-regular rather than convex, reasonable conditions
ensure that the subproblem (\ref{pls})
% PLS($x,\mu$) 
still has a unique local solution close to zero, for $\mu$
sufficiently large, also characterized by property \eqnok{pls_opt}.
Then, by projecting the point $x+d$ onto the inverse image under the
map $c$ of the domain of the function $h$, we can obtain a step that
reduces the objective.

The final part of Section \ref{sec:sized} focuses on the common
situation in which the function $h$ is partly smooth at the point
$\bar c$ relative to a certain manifold $\cM$ --- a generalization of
the surface defined by the active constraints in classical nonlinear
programming.  We give conditions guaranteeing that, when $x$ is close
to $\bar{x}$, the algorithm ``identifies'' $\cM$, in the sense that
the solution $d$ of the subproblem (\ref{gen-prox}) has $\Phi(d) \in
\cM$.

Section~\ref{sec:proxdesc} presents the ProxDescent algorithm in full
generality and proves a global convergence result.  Finally,
Section~\ref{sec:computational} describes some promising preliminary
computational experiments, on convex and nonconvex regularized linear
least-squares problems, together with a polyhedral penalty function
arising from a nonlinear programming application.

\subsection{Variational Analysis Tools} \label{sec:definitions}

We begin with some important basic ideas and notation. We denote by
$P_S(v)$ the usual Euclidean projection of a vector $v \in \R^m$ onto
a closed set $S \subset \R^m$. The {\em distance} between $x$ and the
set $S$ is
\[
\dist (x,S) = \inf_{y \in S} |x-y|.  
\]
We use $B_{\epsilon}(x)$ to denote the closed Euclidean ball of radius
$\epsilon$ around a point $x$.

We write $\bar \R$ for the extended reals $[-\infty,+\infty]$, and
consider a function $h:\R^m \to \bar \R$.  The notion of the
subdifferential of $h$ at a point $\bar c \in \R^m$, denoted $\partial
h(\bar c)$, provides a powerful unification of the classical gradient
of a smooth function and the subdifferential from convex analysis.  It
is a set of generalized gradient vectors, coinciding exactly with the
classical convex subdifferential \cite{Roc70} when $h$ is lower
semicontinuous and convex, and equaling $\{\nabla h(\bar c)\}$ when
$h$ is $\cC^1$ around $\bar c$.  For formal definitions from
variational analysis, we refer the reader to standard texts such as
\cite{Roc98} and \cite{Mor06}.  For ease of reading, we collect such
definitions (along with brief discussions) in an Appendix.

Since the notion of ``prox-regularity'' is crucial for our development,   
%An elegant framework for unifying smooth and convex analysis is
%furnished by the notion of ``prox-regularity'' \cite{Pol96}.
%Geometrically, the idea is rather natural: a set $S \subset \R^m$
%is {\em prox-regular} at a point $s \in S$ if every point near $s$ has
%a unique nearest point in $S$ (using the Euclidean distance).  In
%particular, closed convex sets are prox-regular at every point.  A
%finite collection of $\cC^2$ equality and inequality constraints
%defines a set that is prox-regular at any point where the gradients of
%the active constraints are linearly independent.
%
%The definition of prox-regularity for functions, given by Poliquin and
%Rockafellar~\cite{Pol96}, involved the subdifferential. 
we quote a full definition here, from \cite[Definition~13.27]{Roc98}.
%% For the equivalence with the geometric definition above, see Poliquin,
%% Rockafellar, and Thibault~\cite{Pol00}.

\begin{definition} \label{def:proxreg}
A function $h:\R^m \to \bar \R$ is {\em prox-regular at a point $\bar{c} \in \R^m$ for a
  subgradient $\bar{v} \in \partial h(\bar{c})$} if $h$ is
finite at $\bar{c}$, locally lower semicontinuous around $\bar{c}$,
and there exists $\rho>0$ such that
\[
h(c') \ge h(c) + \langle v, c'-c \rangle - \frac{\rho}{2} |c'-c|^2
\]
whenever points $c,c' \in \R^m$ are near $\bar{c}$ with the value
$h(c)$ near the value $h(\bar{c})$ and for every subgradient $v
\in \partial h(c)$ near $\bar{v}$. Further, $h$ is {\em prox-regular
  at $\bar{c}$} if it is prox-regular at $\bar{c}$ for every $\bar{v}
\in \partial h(\bar{c})$.
\end{definition}

%Note in particular that if $h$ is prox-regular at $\bar{c}$, we have
%that, for every $\bar{v} \in \partial h(\bar{c})$, there exists
%$\rho>0$ such that
%%
%\begin{equation} \label{pr.1}
%h(c') \ge h(\bar{c}) + \langle \bar{v},  c'-\bar{c} \rangle - \frac{\rho}{2} 
%| c'-\bar{c} |^2,
%\end{equation}
%%
%whenever $c'$ is near $\bar{c}$. (Set $c=\bar{c}$ in the definition
%above.)

%% Erroneous: removed.
%% A function $h:\R^m \to \bar \R$ is {\em prox-regular} at a point $\bar
%% c$ if $h(\bar c)$ is finite and the epigraph
%% \[
%% \mbox{epi}\, h ~:=~ \big\{ (c,r) \in \R^m \times \R : r \ge h(c) \big\}
%% \]
%% is prox-regular at the point $\big(\bar c, h(\bar c)\big)$.  In particular,
%% both convex and $\cC^2$ functions are prox-regular wherever they are
%% defined.

\noindent
While this definition appears formidably technical in its generality,
it holds commonly in practice.  Prevalent examples include continuous
functions $h$ with the property that the function $h + \kappa
|\cdot|^2$ is convex for some constant $\kappa$.  For further
discussion, see the Appendix.

A weaker property than the prox-regularity of a function $h$ is
``subdifferential regularity.''  Formal definitions and discussion can
be found in standard texts and in the Appendix.  Here, we simply note
that both $\cC^1$ functions and lower semicontinuous convex functions
are subdifferentially regular, as are sums of such functions.

We next turn to the idea of ``partial smoothness'' introduced by
Lewis~\cite{Lew03a}, a variational-analytic formalization of the
notion of the active set in classical nonlinear programming: see also
Hare and Lewis~\cite[Definition~2.3]{HarL04}.  A set $\cM \subset
\R^m$ is a {\em manifold about} a point $\bar c \in \cM$ if it can be
described locally by a collection of smooth equations with linearly
independent gradients.  More precisely, there exists a map $F:\R^m \to
\R^k$ that is $\cC^2$ around $\bar c$, with $\nabla F(\bar c)$
surjective, such that points $c \in \R^m$ near $\bar c$ lie in $\cM$
if and only if $F(c) = 0$.  The {\em normal space} to $\cM$ at $\bar
c$, denoted $N_{\cM}(\bar c)$ is then just the range of $\nabla F(\bar
c)^*$.

\begin{definition} \label{def:ps}
Given a manifold $\cM \subset \R^m$ about a point $\bar c$, a function
$h:\R^m \to \bar \R$ is {\em partly smooth} at $\bar c$ relative to
$\cM$ if $h$ is subdifferentially regular at all points $c \in \cM$
near $\bar{c}$, the dependence of the value $h(c)$ and the (nonempty)
subdifferential $\partial h(c)$ on the point $c \in \cM$ are $\cC^2$
and continuous respectively, and furthermore the affine span of
$\partial h(\bar{c})$ is a translate of the normal space
$N_{\cM}(\bar{c})$.  We refer to $\cM$ as the {\em active manifold}.
\end{definition}

% \noindent
As with prox-regularity, this definition appears technical.  To
illustrate, consider again the example of continuous function $h$ such
that $h+\kappa |\cdot|^2$ is convex for some $\kappa$.  Since such
functions are always subdifferentially regular, partial smoothness
amounts to smoothness of the restriction $h|_{\cM}$, continuity with
respect to the point $c \in \cM$ of the classical directional
derivative $h'(c;d)$ for all fixed directions $d$, and the property
$h'(\bar c;d) \ge -h'(\bar c;-d)$ for all $d \in N_{\cM}(\bar{c})$.

%A set $S \subset \R^m$ is {\em partly smooth} at a point
%$\bar{c}\in S$ relative to a manifold $\cM$ if its indicator
%function,
%\[
%\delta_{S}(c) = 
%\left\{
%\begin{array}{ll}
%0       & (c \in S) \\
%+\infty & (c \not\in S),
%\end{array}
%\right.
%\]
%is partly smooth at $\bar{c}$ relative to $\cM$.  Again we refer to $\cM$ as the {\em active manifold}.

\section{Examples}
\label{sec:examples}

Our basic framework admits a wide variety of interesting
problems, as we show in this section. 

\subsection{Approximation Problems}

\mbox{}

\begin{Exa}[least squares, $\ell_1$, and Huber approximation]
  The formulation \eqnok{hc} encompasses both the usual (nonlinear)
  least squares problem if we define $h(\cdot) = |\cdot |^2$, and the
  $\ell_1$ approximation problem if we define $h(\cdot) = | \cdot
  |_1$, the $\ell_1$ norm. Another popular robust loss function is the
  Huber function defined by $h(c) = \sum_{i=1}^m \phi(c_i)$, where
\[
\phi(c_i) = \begin{cases}
\half c_i^2      & \mbox{\rm ($|c_i| \le T$)} \\
Tc_i - \half T^2 & \mbox{\rm ($|c_i| > T$)}.
\end{cases}
\]
\end{Exa}

\begin{Exa}[sum of Euclidean norms]
Given a collection of smooth vector functions $g_i: \R^n \to \R^{m_i}$,
for $i=1,2,\dotsc,t$, consider the problem
\[
\min_x \, \sum_{i=1}^t | g_i(x) |.
\]
We can place such problems in the form \eqnok{hc} by defining
the smooth vector function $c:\R^n \to \R^{m_1} \times \R^{m_2} \times
\cdots \times \R^{m_t}$ by $c = (g_1,g_2,\ldots,g_t)$, and the
nonsmooth function $h: \R^{m_2} \times \cdots \times \R^{m_t} \to \R$
by
\[
h(g_1,g_2,\ldots,g_t) = \sum_{i=1}^t | g_i |.
\]
\end{Exa}

This form is seen in facility location problems and in regularized
optimization problems with group-sparse regularizers.

\subsection{Nonlinear Programming Penalty Functions}

Next, we consider examples motivated by penalty functions for
nonlinear programming.

\begin{Exa}[$\ell_1$ penalty function]
\label{ex:l1pf}
Consider the following nonlinear program:
\begin{align}
\label{nlp}
\min \, f(x)                                                                     &                              \\
\nonumber
\mbox{\rm subject to} \;\;  g_i(x)                                               & =0  \quad (1\le i\le j),     \\
\nonumber
g_i(x)                                                                           & \le 0 \quad (j \le i \le k), \\
\nonumber
 x                                                                               & \in X,
\end{align}
where the polyhedron $X \subset \R^n$ describes constraints on the variable $x$ that are easy to handle directly.  The $\ell_1$ penalty function formulation is 
\begin{equation} \label{l1pf}
\min_{x \in X} \, f(x) + \nu \sum_{i=1}^{j} | g_i(x)| + \nu \sum_{i=j+1}^{k} \max \big(0,g_i(x)\big),
\end{equation}
where $\nu>0$ is a scalar parameter.  We can express this problem in
the form \eqnok{hc} by defining the smooth vector function
\[
c(x) ~=~ \Big( f(x) \,,\, \big(g_i(x)\big)_{i=1}^{k} \,,\, x \Big) 
~\in~ \R \times \R^k \times \R^n
\]
and the extended polyhedral convex function $h: \R \times \R^k \times
\R^n \to \bar \R$ by
\[
h(f,g,x) = 
\left\{
\begin{array}{ll}
f + \nu {\displaystyle \sum_{i=1}^{j} |g_i| + \nu \sum_{i=j+1}^{k}} \max (0,g_i) & (x \in X)                    \\
+\infty                                                                          & (x \not\in X).
\end{array}
\right.
\]
\end{Exa}

%% A generalization of Example~\ref{ex:l1pf} in which $h$ is a finite
%% polyhedral function was the focus of much research in the 1980s. We
%% consider this case further in Section~\ref{poly} and use it again
%% during the paper to illustrate the theory that we develop.

\subsection{The Finite Polyhedral Case} \label{poly}

A generalization of the polyhedral convex function of the previous
subsection is obtained by defining 
\begin{equation} \label{eq:hpoly}
h(c) = \max_{i \in I} \{ \ip{h_i}{c} + \beta_i \},
\end{equation}
where $I$ is a finite set of indices, with $h_i \in \R^m$ and $\beta_i
\in \R$ for all $i \in I$.  We return to this case below to illustrate
much of our theory (in Sections~\ref{sec:unique} and \ref{sec:id}, for
example).

Assume that the map $c \colon \R^n \to \R^m$ is $\cC^1$ around a
critical point $\bar x \in \R^n$ for the composite function $h \circ
c$, and let $\bar c = c(\bar x)$.  Define the set of ``active''
indices
\[
\bar I ~=~ \mbox{argmax} \big\{ \ip{h_i}{\bar c} + \beta_i : i \in I \big\}.
\]
Then, denoting convex hulls by $\mbox{conv}$, we have
$
\partial h (\bar c) = \mbox{conv} \{ h_i : i \in \bar I\}.
$
The basic criticality condition \eqnok{hc_necc} becomes existence
of a vector $\lambda \in \R^{\bar I}$ satisfying
\begin{equation} \label{first}
\lambda \ge 0 
~~~\mbox{and}~~~
\sum_{i \in \bar I} \lambda_i 
\left[
\begin{array}{c}
\nabla c(\bar x)^* h_i \\
1
\end{array}
\right]
=
\left[
\begin{array}{c}
0                      \\
1
\end{array}
\right].
\end{equation}
The subgradient $\bar v$ is then $\sum_{i \in \bar I} \lambda_i h_i$.

Compare this condition with the one obtained from the standard
nonlinear programming framework, which is
\beq
\label{classical}
\min_{(x,t) \in  \Re^n \times \Re}  \;   t     \;\;
\mbox{subject to} \;\; \ip{h_i}{c(x)} + \beta_i + t \le 0 \quad (i \in I).
\eeq
%% \begin{alignat}{2} 
%% \nonumber
%% \min_{(x,t) \in  \R^n \times \R}  \;   t &                                                                                     \\
%% \label{classical}
%% \mbox{subject to} \;\;                   &                              & \ip{h_i}{c(x)} + \beta_i + t & \le 0 \quad (i \in I) \\
%% \nonumber
%%                                          &                              & (x,t)                        & \in \R^n \times \R.
%% \end{alignat}
% \begin{equation} \label{classical}
% \left.
% \begin{array}{lrcl}
% \mbox{minimize}                           & t                            &                              &                       \\
% \mbox{subject to}                         & \ip{h_i}{c(x)} + \beta_i + t & \le                          & 0 ~~~(i \in I)        \\
%                                           & (x,t)                        & \in                          & \R^n \times \R.
% \end{array}
% \right\}
% \end{equation}
At the point $\big(\bar x,-h(\bar c)\big)$, the conditions
\eqnok{first} are just the standard first-order optimality conditions,
with Lagrange multipliers $\lambda_i$.  The fact that the vector $\bar
v$ in the criticality condition \eqnok{hc_necc} is closely identified
with $\lambda$ via the relationship $\bar v = \sum_{i \in \bar I}
\lambda_i h_i$ motivates our terminology ``multiplier vector''.

% We return to this example repeatedly below.

\subsection{Regularized Minimization Problems}
\label{sec:regmin}

A large family of instances of \eqnok{hc} arises in the area of
regularized minimization, where the minimization problem has the
following general form:
\begin{equation}
  \label{eq:regmin}
  \min_x \, f(x) + \tau |x|_*
\end{equation}
where $f:\R^n \to \R$ is a smooth objective, while $|x|_*$ is a
continuous, nonnegative, usually nonsmooth function, and $\tau$ is a
nonnegative {\em regularization parameter}. Such formulations arise
when we seek an approximate minimizer of $f$ that is ``simple'' in
some sense; the purpose of the second term $|x|_*$ is to promote this
simplicity property.  Larger values of $\tau$ tend to produce
solutions $x$ that are simpler, but less accurate as minimizers of
$f$. The problem (\ref{eq:regmin}) can be put into the framework
(\ref{hc}) by defining
\begin{equation}
  \label{eq:reghc}
c(x) = \left[
\begin{matrix} f(x) \\ x \end{matrix} \right] \in \R^{n+1}, \qquad
h(f,x) = f + \tau |x|_*.
\end{equation}

We list now some interesting cases of (\ref{eq:regmin}).

\begin{Exa}[$\ell_1$-regularized minimization] \label{ex:l1reg} The
  choice $|\cdot|_*=|\cdot|_1$ in \eqnok{eq:regmin} tends to produce
  solutions $x$ that are {\em sparse}, in the sense of having
  relatively few nonzero components. Larger values of $\tau$ tend to
  produce sparser solutions.  Compressed sensing is a particular area
  of  interest, in which the objective $f$ is typically a
  least-squares function $f(x) = (1/2) | Ax-b |^2$; see \cite{Can06a}
  for a survey.  Regularized least-squares problems (or
  equivalent constrained-optimization formulations) are also
  encountered in statistics; see for example the LASSO \cite{Tib96}
  and LARS \cite{EfronHJT04} procedures, and basis pursuit
  \cite{CheDS98}.

  A related application is regularized logistic regression, where
  again $|\cdot|_* = | \cdot|_1$, but $f$ is (the negative of) an a
  posteriori log likelihood function \cite{ShiWWLKK06}. Here, the
  components of $x$ are weights applied to the features in a data
  vector. We aim to identify those features (corresponding to the
  nonzero locations in $x$) that are most effective in predicting a
  binary outcome.
 %% In the setup of \cite{ShiWWLKK06}, $x$ contains the coefficients
 %%  of a basis expansion of a log-odds ratio function, where each
 %%  basis function is a function of the feature vector. The objective
 %%  $f$ is the (negative) log likelihood function obtained by
 %%  matching this data to a set of binary labels. In this case, $f$
 %%  is convex but highly nonlinear. The regularization term causes
 %%  the solution to have few nonzero coefficients, so the formulation
 %%  identifies the most important basis functions for predicting the
 %%  observed labels.

\end{Exa}

Another interesting class of regularized minimization problems arises
in {\em matrix completion}, where we seek an $m \times n$ matrix $X$
of smallest rank that is consistent with given knowledge of various
linear combinations of the elements of $X$; see
\cite{CanR08,RecFP07a,CaiCS08a}.  Much as the $\ell_1$ norm of a vector $x$
is used as a surrogate for cardinality of $x$ in the formulations of
Example~\ref{ex:l1reg}, the {\em nuclear norm} is used as a surrogate
for the rank of $X$ in formulations of the matrix completion
problem. The nuclear norm $|X|_*$ is defined as the sum of singular
values of $X$, and we have the following specialization of
(\ref{eq:regmin}):
\begin{equation} \label{eq:matcomp}
\min_{X \in \R^{m \times n}} \, \half | \cA (X)-b |^2 + \tau |X|_*,
\end{equation}
where $\cA$ denotes a linear operator from $\R^{m \times n}$ to
$\R^p$, and $b \in \R^p$ is the observation vector. Note that the
nuclear norm is a continuous and convex function of $X$. 

Finally, we mention image denoising and deblurring problems, which are
often posed in the form \eqnok{eq:regmin}, where $| \cdot|_*$ is a
total-variation regularizer \cite{ROF92} that induces ``natural''
qualities in the solution images. Specifically, the recovered images
contain large areas of near-constant color or shade, separated by
sharp edges.

For regularized minimization problems of the form (\ref{eq:regmin}),
the subproblem (\ref{pls}) has the form
\begin{equation}
  \label{eq:pls.regmin}
  \min_d \, f(x) + \langle \nabla f(x),d \rangle + \frac{\mu}{2} |d|^2 
+ \tau |x+d|_*.
\end{equation}
An equivalent formulation can be obtained by shifting the objective
and making the change of variable $z:=x+d$:
\begin{equation}
  \label{eq:pls.regmin.3}
  \min_z \, \frac{\mu}{2} |z  - y|^2 + \tau |z|_*, 
\qquad \mbox{\rm where} \quad 
y = x - \frac{1}{\mu} \nabla f(x).
\end{equation}
When the regularization function $|\cdot|_*$ is separable in the
components of $x$, as when $|\cdot|_*=|\cdot|_1$ or $|\cdot| = |
\cdot|_2^2$, this problem can be solved in $O(n)$ time.  (This fact is
key to the practical efficiency of methods based on these subproblems
in compressed sensing; see \cite{WriNF08a}.) For the case $| \cdot|_*
= | \cdot|_1$, if we set $\alpha = \tau/\mu$, the solution of
(\ref{eq:pls.regmin.3}) is
\begin{equation} \label{eq:shrink}
z_i = \begin{cases}
0            & \mbox{\rm ($|y_i| \le \alpha$)} \\
y_i - \alpha & \mbox{\rm ($y_i > \alpha$)}     \\
y_i + \alpha & \mbox{\rm ($y_i < -\alpha$)}.
\end{cases}
\end{equation}
This operation is known commonly as the ``shrink operator.''

For  matrix completion  (\ref{eq:matcomp}), the
formulation (\ref{eq:pls.regmin.3}) of the subproblem becomes
\begin{equation} \label{eq:matcomp.subprob}
\min_{Z \in \R^{m \times n}} \, \frac{\mu}{2} | Z-Y |_F^2 + \tau |Z|_*,
\end{equation}
where  $| \cdot|_F$ denotes the Frobenius norm of a matrix and 
\begin{equation} \label{defY}
Y = X - \frac{1}{\mu} \cA^* [ \cA(X)-b].
\end{equation}
It is known (see for example \cite{CaiCS08a}) that (\ref{eq:matcomp.subprob})
can be solved by using the singular-value decomposition of
$Y$. Writing $Y = U \Sigma V^T$, where $U$ and $V$ are orthogonal and
$\Sigma = \mbox{\rm diag} (\sigma_1,\sigma_2, \dotsc,
\sigma_{\min(m,n)})$, we have $Z = U \Sigma_{\tau/\mu} V^T$, where the
diagonals of $\Sigma_{\tau/\mu}$ are $\max(\sigma_i - \tau/\mu,0)$ for
$i=1,2,\dotsc, \min(m,n)$. In essence, we apply the shrink operator to
the singular values of $Y$, and reconstruct $Z$ by using the
orthogonal matrices $U$ and $V$ from the decomposition of $Y$.

% \begin{Exa}[composite nonsmooth minimization] \label{gen-poly}
%   Considerable effort in the 1980s was devoted to algorithms for
%   problem \eqnok{hc} for general finite polyhedral convex functions
%   $h$.  This class encompasses some of the examples mentioned above.
%   We discuss it at greater length in the following section, and use
%   it to illustrate the theory throughout.
% \end{Exa}

\subsection{Nonconvex Problems}

Each of the examples above involves a convex outer function $h$.  In
principle, however, the techniques we develop here also apply to a
variety of nonconvex functions. This section discusses some
applications in which $h$ is nonconvex.
%% The next example includes some simple illustrations.

\begin{Exa}[problems involving quadratics] 
  Given a general quadratic function $f: \R^p \to \R$ (possibly
  nonconvex) and a smooth function $c_1 : \R^n \to \R^p$, consider the
  problem $\min_x\, f\big(c_1(x)\big)$.  This problem trivially fits into the
  framework \eqnok{hc}, and the function $f$, being $\cC^2$, is
  everywhere prox-regular.  The subproblems \eqnok{gen-prox}, for
  sufficiently large values of the parameter $\mu$, simply amount to
  solving a linear system.

  More generally, given another general quadratic function $g : \R^q
  \to \R$, and another smooth function $c_2 : \R^n \to \R^q$, consider
  the problem
\[
\min_{x \in \R^n} \, f\big(c_1(x)\big) \quad \mbox{\rm subject to} \quad 
g\big(c_2(x)\big) \le 0.
\]
We can express this problem in the form \eqnok{hc} by defining the
smooth vector function $c = (c_1,c_2)$ and defining an extended-valued
nonconvex function
\[
h(c_1,c_2) = 
\left\{
\begin{array}{ll}
f(c_1)  & (g(c_2) \le 0) \\
+\infty & (g(c_2) > 0).
\end{array}
\right.
\] 
The epigraph of $h$ is
\[
\big\{ (c_1,c_2,t) : g(c_2) \le 0,~ t \ge f(c_1) \big\},
\]
a set defined by two smooth inequality constraints: hence $h$ is
prox-regular at any point $(c_1,c_2)$ satisfying $g(c_2) \le 0$ and
$\nabla g(c_2) \ne 0$.  The resulting subproblems \eqnok{gen-prox} are
all in the form of the standard trust-region subproblem, and hence
relatively straightforward to solve quickly.

As one more example of this type, we consider the case in which the
outer function $h$ is defined as the maximum of a finite collection of
quadratic functions (possibly nonconvex): $h(x) = \max \{f_i(x) :
i=1,2,\dotsc,k \}$.  The subproblems \eqnok{gen-prox} are as follows:
\[
\min \Big\{ t : t \ge f_i\big(\Phi(d)\big) + \frac{\mu}{2}|d|^2,~ d \in \R^m,~ t \in \R,~ i=1,2,\dotsc, k \Big\}.
\]
where the map $\Phi$ is affine.  For sufficiently large values of the
parameter $\mu$, this is a quadratically-constrained convex quadratic
program, which can in principle be solved efficiently by an interior point
method.
\end{Exa}

To conclude, we consider three more nonconvex examples.  The first,
due to Mangasarian~\cite{Man99a}, is used by Jokar and
Pfetsch~\cite{JokP07a} to find sparse solutions of underdetermined
linear equations. The formulation of \cite{JokP07a} can be stated in
the form (\ref{eq:regmin}) where the regularization function
$|\cdot|_*$ has the form
\[
|x|_* = \sum_{i=1}^n (1-e^{-\alpha |x_i|})
\] 
for some parameter $\alpha>0$. It is easy to see that this function is
nonconvex but prox-regular, and nonsmooth only at $x_i=0$.

Fan and Li~\cite{SCAD} propose the smoothly clipped absolute deviation
(SCAD) regularizer. This problem has the form (\ref{eq:regmin}), and
behaves like the $\ell_1$ norm near the origin, transitioning (via a
concave quadratic) to a constant for large loss values. Specifically,
we have $|\cdot|_* = \sum_{i=1}^n \phi(x_i)$, where
\[
\phi(x_i) = \begin{cases}
\lambda |x_i|                                               & \mbox{\rm ($|x_i| \le \lambda$)}   \\
-(|x_i|^2 - 2a\lambda |x_i| + \lambda^2) / \big(2(a-1)\big) & 
\mbox{\rm ($\lambda < |x_i| \le a\lambda$)}                                                      \\
(a+1)\lambda^2/2                                            & \mbox{\rm ($|x_i| > a\lambda$)}.
\end{cases}
\]
Here $\lambda>0$ and $a>1$ are tuning parameters. The minimum concave
penalty (MCP) regularizer of Zhang~\cite{MCP} has a similar form, with 
\begin{equation} \label{eq:mcp}
\phi(x_i) = \begin{cases}
\lambda |x_i| - |x_i|^2/(2a)                                & \mbox{\rm ($|x_i| \le a \lambda$)} \\
a \lambda^2/2                                               & \mbox{\rm ($|x_i|>a \lambda$)}.
\end{cases}
\end{equation}
SCAD and MCP have been shown to avoid the bias property associated
with the $\ell_1$ penalty function, in which nonzero values of $x$ are
skewed toward zero.

\section{Related Work} \label{sec:related}

We discuss here some connections of our approach with existing
literature.

\paragraph{Convex $h$} 
Burke~\cite{Bur85a} uses a similar composite function to the one
analyzed here, and a subproblem like \eqnok{gen-prox} to calculate the
search direction $d$. In contrast to our approach, the analysis in
\cite{Bur85a} is restricted to finite convex $h$, and the algorithm
uses a backtracking line search to ensure descent in the composite
objective at each iteration. In place of the prox term $|d|^2/2$ of
\eqnok{gen-prox}, Burke uses ``casting functions'' that serve a
similar purpose of ensuring well posedness of the subproblem. 
% A global convergence result is proved.
Sagastiz\'abal~\cite{Sag13} considers the problem \eqref{hc} in which
$h$ is finite, convex, and positively homogeneous. Her algorithm is
based on a subproblem like \eqref{pls}, differing mainly in that $h$
is replaced by a lower-bounding bundle
approximation. Lan~\cite[Section~4]{Lan15} discusses \eqref{hc} in
which $h$ and the components of $c(x)$ are all Lipschitz continuous
and convex. Under certain assumptions on the smoothness of $c$, a
subproblem is defined that makes use of an approximation like the
$h\big(c(x)+\nabla c(x)d\big)$ of \eqref{pls}, but taking the maximum
of such approximations over all previous iterates, not just the one
from the latest iterate. Global convergence is proved
\cite[Proposition~1]{Lan15} at rates that are optimal among
first-order schemes.

\paragraph{Polyhedral $h$}
Various approaches have been proposed for the case of $h$ finite and
polyhedral.  One work closely related to ours is by Fletcher and Sainz
de la Maza~\cite{FleS89}, who discuss an algorithm for minimization of
the $\ell_1$ penalty function \eqnok{l1pf} for the nonlinear
optimization problem \eqnok{nlp}. At each iteration, their method
solves a linearized trust-region problem that can be expressed in our
general notation as follows:
\begin{equation} \label{fsm.lin}
\min_{d} \, h\big(c(x)+ \nabla c(x) d\big)  \;\; \makebox{\rm subject to} \;\;
| d | \le \rho,
\end{equation}
where $\rho$ is some trust-region radius. Note that this subproblem is
closely related to our linearized subproblem \eqnok{pls} when the
Euclidean norm is used to define the trust region. However, the
$\ell_{\infty}$ norm is preferred in \cite{FleS89}, as it allows the
subproblem \eqnok{fsm.lin} to be expressed as a linear program. The
algorithm in \cite{FleS89} uses the solution of \eqnok{fsm.lin} to
estimate the active constraint manifold, then computes a step that
minimizes a model of the Lagrangian function for \eqnok{nlp} while
fixing the identified constraints as equalities.  An active-constraint
identification result is proved (\cite[Theorem~2.3]{FleS89}); this
result is related to our Theorems~\ref{th:id1} and \ref{th:id.conv}
below.

% The solution of \eqnok{fsm.lin} is used to construct a benchmark to
% determine acceptability of this (higher-order) step; if it proves
% unacceptable, the solution of \eqnok{fsm.lin} is itself tried. If
% this too is unacceptable, the trust-region radius $\rho$ is
% decreased and the process is repeated. We note, however, that it is
% unclear whether the global convergence theorem in this paper
% \cite[Theorem~2.1]{FleS89} is proved correctly. There is also a
% result \cite[Theorem~2.3]{FleS89} concerning identification of the
% active constraints; our analysis in Section~\ref{sec:id} is in a
% sense a generalization of this result.

Byrd et al.~\cite{ByrGNW05} describe a successive linear-quadratic
programming method, based on \cite{FleS89}, which starts with solution
of the linear program \eqnok{fsm.lin} (with $\ell_{\infty}$ trust
region) and uses it to define an approximate Cauchy point, then
approximately solves an equality-constrained quadratic program (EQP)
over a different trust region to enhance the step. This algorithm is
implemented in the KNITRO package for nonlinear optimization as the
KNITRO-ACTIVE option.

Friedlander et al.~\cite{FriGLM07} solve a problem of the form
(\ref{pls}) for the case of nonlinear programming, where $h$ is the
sum of the objective function $f$ and the indicator function for the
equalities and the inequalities defining the feasible region. The
resulting step can be enhanced by solving an EQP.

Other related literature on composite nonsmooth optimization problems
with general finite polyhedral convex functions (Section~\ref{poly})
includes the papers of Yuan~\cite{Yua85,Yua85b} and
Wright~\cite{Wri90b}. The approaches in \cite{Yua85b,Wri90b} solve a
linearized subproblem like \eqnok{fsm.lin}, from which an analog of
the ``Cauchy point'' for trust-region methods in smooth unconstrained
optimization can be calculated. This calculation involves a line
search along a piecewise quadratic function and is therefore more
complicated than the calculation in \cite{FleS89}, but serves a similar
purpose, namely as the basis of an acceptability test for a step
obtained from a higher-order model.

\paragraph{Regularized Form (\protect\ref{eq:regmin})}
For general outer functions $h$, the theory is more complex.  An early
approach to regularized minimization problems of the form
\eqref{eq:regmin} for a lower semicontinuous convex function
$|\cdot|_*$ is due to Fukushima and Mine~\cite{Fuk81}. They calculate
a trial step at each iteration by solving the linearized problem
\eqref{eq:pls.regmin}.

Subproblems of the form \eqref{eq:pls.regmin} were used in compressed
sensing algorithms by Wright, Nowak, and Figueiredo~\cite{WriNF08a}
and Hale, Yin, and Zhang~\cite{HalYZ07b}, in conjunction with an
adaptive strategy for choosing $\mu$. (Indeed, this application
provided the  motivation for the current study.) 

Combettes and Wajs~\cite{ComW05a} study formulations similar to
\eqnok{eq:regmin} and algorithms that use subproblems like
\eqref{eq:pls.regmin}. Apart from assuming convexity, their setting is
more general. Convergence is proved for algorithms that use values of
$\mu$ in \eqref{eq:pls.regmin} that are large enough to guarantee
descent in the objective at every iteration, regardless of iterate
$x$. This assumption contrasts with the adaptive approach used in
\cite{WriNF08a} and in Section~\ref{sec:proxdesc} below.

\paragraph{$c(x)=x$: Proximal-Point Methods}
The case when the map $c$ is simply the identity has a long history.
The iteration $x_{k+1} = x_k + d_k$, where $d_k$ minimizes the
function $d \mapsto h(x_k+d) + \frac{\mu}{2}|d|^2$, is the well-known
{\em proximal point method}.  For lower semicontinuous convex
functions $h$, convergence was proved by Martinet~\cite{Mar70} and
generalized by Rockafellar~\cite{Roc76}.  For nonconvex $h$, a good
survey up to 1998 is by Kaplan and Tichatschke~\cite{Kap98}.
Pennanen~\cite{Pen02} took an important step forward, showing in
particular that if the graph of the subdifferential $\partial h$
agrees locally with the graph of the inverse of a Lipschitz function
(a condition verifiable using second-order properties including
prox-regularity---see Levy~\cite[Cor.~3.2]{Lev01}), then the proximal
point method converges linearly if started nearby and with
regularization parameter $\mu$ bounded away from zero.  This result
was foreshadowed in much earlier work of Spingarn~\cite{Spi81}, who
gave conditions guaranteeing local linear convergence of the proximal
point method for a function $h$ that is the sum of lower
semicontinuous convex function and a $\cC^2$ function, conditions
which furthermore hold ``generically'' under perturbation by a linear
function. Inexact variants of Pennanen's approach are discussed by
Iusem, Pennanen, and Svaiter~\cite{Ius03} and Combettes and
Pennanen~\cite{Com04}.  In this current work, we make no attempt to
build on this more sophisticated theory, preferring a more direct and
self-contained approach.

\paragraph{Manifold Identification}
The issue of identification of the face of a constraint set on which
the solution of a constrained optimization problem lies has been the
focus of numerous works. For the problem $\min_{x \in X} \, f(x)$, for
a closed set $X \subset \R^n$, some papers show that the projection of
the point $x-(1/\mu) \nabla f(x)$ onto the feasible set $X$ (for some
fixed $\mu>0$) lies on the same face as the solution $\bar{x}$, under
certain nondegeneracy assumptions on the problem and geometric
assumptions on $X$.  Identification of so-called quasi-polyhedral
faces of convex $X$ was described by Burke and Mor\'e~\cite{BurM88}.
% Work on identification of active
% faces of convex sets was presented by Burke and
% Mor\'e~\cite{BurM88}. It was shown that so-called quasi-polyhedral
% faces of such sets could be identified finitely by algorithms of
% gradient projection type, under appropriate nondegeneracy
% conditions. (By ``finite identification,'' we mean that the sequence
% of iterates $\{x_k\}$ generated by the algorithm is such that $x_k$
% lies in the appropriate set for all $k$ sufficiently large.) 
An extension to the nonconvex case is provided by Burke~\cite{Bur90},
who considers algorithms that work with linearizations of the
constraints describing $X$. Wright~\cite{Wri93a} considers surfaces of
a convex set $X$ that can be parametrized by a smooth algebraic
mapping, and shows how algorithms of gradient projection type can
identify such surfaces once the iterates are sufficiently close to a
solution.  Lewis~\cite{Lew03a} and Hare and Lewis~\cite{HarL04} extend
these identification results to the nonconvex, nonsmooth case by using
concepts from nonsmooth analysis, including partly smooth functions
and prox-regularity. In their setting, the concept of an identifiable
face of a feasible set is extended to a certain type of manifold with
respect to which the function $h$ in \eqnok{hc} is partly smooth (see
Definition~\ref{def:ps} above). 
%% Their main results give conditions under which the active manifold
%% is identified from within a neighborhood of the solution.

A rich class of convex composite functions with partly smooth
structure was discussed in detail by Bonnans and Shapiro~\cite{BonS00}
and Shapiro~\cite{Sha03c}.  For a detailed discussion of the
relationship between that class and partial smoothness, see
\cite{BolDL11a}.

\paragraph{Alternative Subproblems}
Another line of relevant work is associated with the $\cV \cU$ theory
introduced by Lemar\'echal, Oustry, and Sagastiz\'abal~\cite{LemOS00}
and subsequently elaborated by these and other authors.  The focus is
on minimizing convex functions $f(x)$ that, again, are partly smooth
--- smooth (``U-shaped'') along a certain manifold through the
solution $\bar{x}$, but nonsmooth (``V-shaped'') in the transverse
directions.  Mifflin and Sagastiz\'abal~\cite{SagM02} discuss the
``fast track,'' which is essentially the manifold containing the
solution $\bar{x}$ along which the objective is smooth. Similarly to
\cite{FleS89}, they are interested in algorithms that identify the
fast track and then take a minimization step for a certain Lagrangian
function along this track. It is proved in \cite[Theorem~5.2]{SagM02}
that under certain assumptions, when $x$ is near $\bar{x}$, the
proximal point $x+d$ obtained by solving the problem
\begin{equation} \label{prox.gen} 
\min_d \, f(x+d) + \frac{\mu}{2} |d|^2 
\end{equation}
lies on the fast track. This identification result is similar to the
one we prove in Section~\ref{sec:id}, but the calculation of $d$ is
different.
% in \eqnok{prox.gen} from our approach via subproblem \eqnok{pls}.  
In our case of $f=h \circ c$, \eqnok{prox.gen} becomes
\begin{equation} \label{prox} 
\min_d \, h\big(c(x+d)\big) + \frac{\mu}{2} |d|^2.
\end{equation}
%% whose optimality conditions are, for some fixed current iterate $x$, 
%% \begin{equation} \label{prox_opt}
%% \nabla c(x+d)^* v + \mu d = 0, \;\; 
%% \makebox{\rm for some $v \in \partial h \big(c(x+d)\big)$}.
%% \end{equation}
%% %
%% Compare this system with the optimality conditions (\ref{pls_opt})
%% from subproblem (\ref{pls}):
%% \[
%% \nabla c(x)^* v + \mu d = 0, \;\; 
%% \makebox{\rm for some $v \in \partial h\big(c(x) + \nabla c(x) d\big)$}.
%% \]
In many applications of interest, $c$ is nonlinear, so the subproblem
\eqnok{prox} is generally harder to solve for the step $d$ than our
subproblem \eqnok{pls}.

Mifflin and Sagastiz\'abal~\cite{MifS05a} describe an algorithm in
which an approximate solution of \eqnok{prox.gen} is obtained, again
for the case of a convex objective, by making use of a piecewise
linear underapproximation to their objective $f$, usually constructed
from a bundle of subgradients gathered at earlier iterations.
%% The approach is most suitable for a bundle method in which the
%% piecewise-linear approximation is constructed from subgradients
%% gathered at previous iterations.
Approximations to the manifold of smoothness for $f$ are constructed,
and a Newton-like step for the Lagrangian is taken along this
manifold.  Daniilidis, Hare, and Malick~\cite{DanHM06} use the
terminology ``predictor-corrector'' to describe algorithms of this
type.
%% Their ``predictor'' step is the step along the manifold of
%% smoothness for $f$, while the ``corrector'' step (\ref{prox.gen})
%% eventually returns the iterates to the correct active manifold (see
%% \cite[Theorem~28]{DanHM06}).
Miller and Malick~\cite{MilM05} show how
algorithms of this type are related to Newton-like methods that have
been proposed earlier in various contexts.

Various of the algorithms discussed above make use of curvature
information for the objective on the active manifold to accelerate
local convergence. The algorithmic framework that we describe in
Section~\ref{sec:proxdesc} can be modified to incorporate similar
techniques, while retaining its global convergence and manifold
identification properties. Algorithms with this flavor have been
described in \cite{ShiWWLKK06} for the case of $\ell_1$-regularized
logistic regression, and \cite{WenYZG08} for $\ell_1$-regularized
least squares.

\section{Properties of the Proximal Linearized Subproblem}
\label{sec:sized}

We show in this section that when $h$ is prox-regular at $\bar{c}$,
under a mild additional assumption, the subproblem (\ref{pls})
% PLS($x,\mu$) 
has a local solution $d$ with norm $O(| x-\bar{x} |)$, when the
parameter $\mu$ is sufficiently large. When $h$ is convex, this
solution is the unique global solution of the subproblem. We show too
that a point $\xnew$ near $x+d$ can be found such that the objective
value $h\big(c(\xnew)\big)$ is close to the prediction of the model
function $h(c(x)+\nabla c(x) d)$ from \eqnok{pls}. Further, we
describe conditions under which the subproblem correctly identifies
the manifold ${\cal M}$ with respect to which $h$ is partly smooth at
the solution of (\ref{hc}).

\subsection{Lipschitz Properties}
\label{preliminaries}

We start with technical preliminaries.  Allowing non-Lipschitz or
extended-valued outer functions $h$ in our problem (\ref{hc}) is
conceptually appealing, since it allows us to model constraints that
must be enforced.  However, this flexibility presents certain
technical challenges, which we now address.  We begin with a simple
example, to illustrate some of the difficulties.

\begin{Exa}
Define a $\cC^2$ function $c:\R \to \R^2$ by $c(x) = (x,x^2)$, and a lower semicontinuous convex function $h:\R^2 \to \bar\R$ by
\[
h(y,z) = 
\left\{
\begin{array}{ll}
y       & (z \ge 2y^2) \\
+\infty & (z <   2y^2).
\end{array}
\right.
\]
The composite function $h \circ c$ is simply $\delta_{\{0\}}$, the
indicator function of $\{0\}$.  This function has a global minimum
value zero, attained uniquely by $\bar x = 0$.

At any point $x \in \R$, the derivative map $\nabla c(x): \R \to \R^2$
is given by $\nabla c(x) d = (d,2xd)$ for $d \in \R$.  Then, for all
nonzero $x$, it is easy to check that
\[
h\big( c(x) + \nabla c(x) d \big) = + \infty ~~\mbox{for all}~d \in \R,
\]
so the corresponding proximal linearized subproblem  (\ref{pls})
% $\mbox{\rm PLS}(x,\mu)$  defined by equation \eqnok{pls} 
has no feasible solutions: its
objective value is identically $+\infty$.

The adjoint map $\nabla c(0)^* : \R^2 \to \R$ is given by $\nabla c(0)^*v = v_1$ for $v \in \R^2$, and 
\[
\partial h(0,0) ~=~ \big\{v \in \R^2 : v_1 =1,~v_2 \le 0\big\}.
\]
Hence the criticality condition \eqnok{hc_necc} has no solution $\bar v \in \R^2$.
\end{Exa}

This example illustrates two fundamental difficulties.  The first is
theoretical: the basic criticality condition \eqnok{hc_necc} may be
unsolvable, essentially because the chain rule fails.  The second is
computational: if, implicit in the function $h$, are constraints on
acceptable values for $c(x)$, then curvature in these constraints can
cause infeasibility in linearizations.  As we see below, resolving
both difficulties requires a kind of ``transversality'' condition
common in variational analysis.

In this section we make use of the {\em normal cone} to a set $S$ at a
point $s \in S$, denoted by $N_S(s)$, defined in the Appendix.  When $S$ is convex, it coincides exactly
with the classic normal cone from convex analysis, while for smooth
manifolds it coincides with the classical normal space.

The transversality condition we need involves the ``horizon
subdifferential'' of the function $h : \R^m \to \bar\R$ at the point
$\bar c \in \R^m$, denoted $\partial^{\infty} h(\bar c)$.  This
object, which recurs throughout our analysis, consists of a set of
``horizon subgradients'', capturing information about directions in
which $h$ grows faster than linearly near $\bar c$.  (See
the Appendix for a formal definition.)   This idea simplifies in
important special cases.  If $h$ is convex, finite, and lower
semicontinuous at $\bar c$, we have the following relationship between
the subdifferential and the classical normal cone to the
domain (see \cite[Proposition~8.12]{Roc98}):
$
\partial ^{\infty} h(\bar c) = N_{\mbox{\scriptsize dom}\,h}(\bar c). 
$
We have further that
$\partial^{\infty} h(\bar c) = \{0\}$ if $h$ is locally Lipschitz around $\bar c$.

This condition holds in particular for a convex function $h$ that is
continuous at $\bar c$.
%%  Readers interested only in continuous convex functions $h$ may
%% therefore make the substantial simplification $\partial^{\infty}
%% h(\bar c) = \{0\}$ throughout the analysis.

We seek conditions guaranteeing a reasonable step in the proximal
linearized subproblem \eqnok{pls}.  Our key tool is the following
technical result.

\begin{theorem} \label{th:pert}
Consider a lower semicontinuous function $h \colon \R^m \rightarrow \bar \R$,
a point $\bar z \in \R^m$ where $h(\bar z)$ is finite, and a linear
map $\bar G \colon \R^n \rightarrow \R^m$ satisfying
\[
\partial ^{\infty} h(\bar z) \cap \mbox{\rm Null}(\bar G^*) = \{0\}.
\]
Then there exists a constant $\gamma > 0$ such that, for all vectors
$z \in \R^m$ and  linear maps $G : \R^n \to \R^m$ with $(z,G)$ near $(\bar z,\bar G)$,
there exists a vector $w \in \R^n$ satisfying
\[
|w| \le \gamma |z-\bar{z}| ~~\mbox{and}~~ h(z+Gw) \le h(\bar z) + \gamma |z-\bar{z}|.
\]
\end{theorem}
Notice that this result is trivial if $h$ is locally Lipschitz (or in
particular continuous and convex) around $\bar z$, since we can simply
choose $w=0$.  The non-Lipschitz case is harder; our proof appears
below following the introduction of a variety of ideas from
variational analysis whose use is confined to this subsection.  We
refer the reader to Rockafellar and Wets~\cite{Roc98} or
Mordukhovich~\cite{Mor06} for further details. First, we need a
``metric regularity'' result, which is proved below by means of a
result from Dontchev, Lewis, and Rockafellar~\cite{Don03}. An
alternative proof, which sets the result in a broader context, appears
in the Appendix.

%% Since this theorem is a fundamental tool for us, we give two
%% proofs, one of which specializes the proof of Theorem 3.3 in
%% Dontchev, Lewis and Rockafellar~\cite{Don03}, while the other sets
%% the result into a broader context.

\begin{theorem}[uniform metric regularity under perturbation] 
\label{metric}
\hfill \mbox{} \mbox{Suppose that the} closed set-valued mapping $F
\colon \R^p \tto \R^q$ is metrically regular at a point $\bar u \in
\R^p$ for a point $\bar v \in F(\bar u)$: in other words, there exist
positive constants $\kappa$ and $a$ such that all points $u \in
B_a(\bar u)$ and $v \in B_a(\bar v)$ satisfy
\begin{equation} \label{unperturbed}
\dist\!\big(u,F^{-1}(v)\big) \le \kappa \, \dist\!\big(v,F(u)\big).
\end{equation}
Then there exist constants $\delta,\gamma > 0$ such that all linear
maps $H \colon \R^p \rightarrow \R^q$ with $\|H\| < \delta$ and all
points $u \in B_{\delta}(\bar u)$ and $v \in B_{\delta}(\bar v)$
satisfy
\begin{equation} \label{perturbed}
\dist\!\big(u,(F+H)^{-1}(v)\big) ~\le~ \gamma \, \dist\!\big(v,(F+H)(u)\big).  
\end{equation}
\end{theorem}
\begin{proof}
%%  For our first approach, we 
We follow the notation of the proof of \cite[Theorem~3.3]{Don03}.  Fix
any constants
% \begin{align*}
% \lambda & \in (0 , \kappa^{-1})                                               \\
% \alpha  & \in \big( 0 , \frac{a}{4}(1- \kappa \lambda) \min\{1,\kappa\} \big) \\
% \delta  & \in \big( 0 , \min \big\{ \frac{\alpha}{4} , \frac{\alpha}{4\kappa} , \lambda \big\} \big).
% \end{align*}
\[
\lambda \in (0 , \kappa^{-1}), \qquad
\alpha  \in \left( 0 , \frac{a}{4}(1- \kappa \lambda) \min\{1,\kappa\} \right), \qquad
\delta  \in \left( 0 , \min \left\{ \frac{\alpha}{4} , \frac{\alpha}{4\kappa} , \lambda \right\} \right).
\]
Then the proof shows inequality (\ref{perturbed}), if we define
$\gamma = {\kappa}/{(1-\kappa \lambda)}$.
\end{proof}

Using this result, and given a closed set $S$ containing $0$, we
identify a condition under which any vector $v$ can be projected to
$S$ along the range space of a given matrix, with the difference
between $v$ and its projection being bounded in terms of $|v|$.  We prove this result in the Appendix.
\begin{corollary} \label{co:pert}
Consider a closed set $S \subset \R^q$ with $0 \in S$, and a linear map
$\bar A \colon \R^p \rightarrow \R^q$ satisfying
\[
N_S(0) \cap \mbox{\rm Null}(\bar A^*) = \{0\}.
\]
Then there exists a constant $\gamma > 0$ such that, for all vectors
$v \in \R^q$ and linear maps $A : \R^p \to \R^q$ with $(v,A)$ near
$(0,\bar A)$, the inclusion
\[
v+Au \in S
\]
has a solution $u \in \R^p$ satisfying $|u| \le \gamma |v|$.
\end{corollary}

%\begin{proof}
%Corresponding to any linear map $A \colon \R^p \rightarrow \R^q$, define a
%set-valued mapping $F_A \colon \R^p \tto \R^q$ by $F_A(u) = Au-S$.  A
%coderivative calculation shows, for vectors $v \in \R^p$,
%\[
%D^* F_A(0|0)(v) =
%\left\{
%\begin{array}{ll}
%\{A^*v\}  & \big(v \in N_S(0)\big) \\
%\emptyset & (\mbox{otherwise}).
%\end{array}
%\right.
%\]
%Hence, by assumption, the only vector $v \in \R^p$ satisfying $0 \in
%D^* F_{\bar A}(0|0)(v)$ is zero, so by \cite[Thm 9.43]{Roc98}, the
%mapping $F_{\bar A}$ is metrically regular at zero for zero.  Applying
%the preceding theorem shows that there exist constants $\delta,\gamma
%> 0$ such that, if $\|A-\bar A\| < \delta$ and $|v| < \delta$, then we have
%\[
%\dist\!\big(0,F_A^{-1}(-v)\big) \le \gamma \, \dist\!\big(-v,F_A(0)\big),
%\]
%or equivalently,
%\[
%\dist\!\big(0,A^{-1}(S-v)\big) \le \gamma \, \dist (v,S).
%\]
%Since $0 \in S$, the right-hand side is bounded above by $\gamma |v|$,
%so the result follows.  
%\end{proof}

We are now ready to prove the main result of this subsection.

{\em Proof of Theorem~\ref{th:pert}.}  Let $S \subset \R^m \times \R$
be the epigraph of $h$, and define a map $\bar A \colon \R^n \times \R
\rightarrow \R^m \times \R$ by $\bar A(z,\tau) = (\bar G z,\tau)$.
From $\partial^{\infty} h(\bar{z}) \cap \mbox{Null} (\bar{G}^*) =
\{0\}$, we have $\mbox{Null}(\bar A^*) = \mbox{Null}(\bar G^*) \times
\{0\}$, so \cite[Theorem~8.9]{Roc98} shows that
\[
N_S(\bar{z},h(\bar{z})) \cap \mbox{\rm Null}(\bar A^*) = \{(0,0)\}.
\]
For any vector $z$ and linear map $G$ with $(z,G)$ near $(\bar z,\bar G)$,
the vector $(z,0) \in \R^m \times \R$ is near the vector $(\bar z,0)$
and the map $(w,\tau) \mapsto (Gw,\tau)$ is near the map $(w,\tau)
\mapsto (\bar G w,\tau)$.  The previous corollary shows the existence
of a constant $\gamma > 0$ such that, for all such $z$ and $G$, the
inclusion
\[
(z,0) + (Gw,\tau) \in S
\]
has a solution satisfying $|(w,\tau)| \le \gamma|(z-\bar{z},0)|$, and the
result follows.~~~ $\Box$

We end this subsection with another tool to be used later, whose proof (in the Appendix)
is a straightforward application of standard ideas from variational
analysis.  Like Theorem \ref{metric}, this tool concerns metric
regularity, this time for a constraint system of the form $F(z) \in S$
for an unknown vector $z$, where the map $F$ is smooth, and $S$ is a
closed set.

\begin{theorem}[metric regularity of constraint
  systems] \label{constraint} 
  Consider a $\cC^1$ map $F \colon \R^p
  \to \R^q$, a point $\bar z \in \R^p$, and a closed set $S \subset
  \R^q$ containing the vector $F(\bar z)$.  Suppose the condition
\[
N_S\big( F(\bar z) \big) \cap \mbox{\rm Null}(\nabla F(\bar z)^*) ~=~ \{0\}
\]
holds.  Then there exists a constant $\kappa > 0$ such that all points $z \in \R^p$ near $\bar z$ satisfy the inequality
\[
\mbox{\rm dist} \big(z,F^{-1}(S)\big) \le \kappa \, \mbox{\rm dist}(F(z),S).
\]
\end{theorem}

%\begin{proof}
%  We simply need to check that the set-valued mapping $G \colon \R^p
%  \tto \R^q$ defined by $G(z) = F(z) - S$ is metrically regular $\bar
%  z$ for zero.  Much the same coderivative calculation as in the proof
%  of Corollary~\ref{co:pert} shows, for vectors $v \in \R^p$, the formula
%\[
%D^* G(\bar z|0)(v) =
%\left\{
%\begin{array}{ll}
%\{\nabla F(\bar z)^*v\} & \big(v \in N_S(\bar z)\big) \\
%\emptyset               & (\mbox{otherwise}).
%\end{array}
%\right.
%\]
%Hence, by assumption, the only vector $v \in \R^p$ satisfying 
%$0 \in D^* G(\bar z|0)(v)$ is zero, so metric regularity follows by
%\cite[Thm 9.43]{Roc98}.
%\end{proof}

\subsection{The Proximal Step}

We now prove a key result.  Under a standard transversality condition,
and assuming the proximal parameter $\mu$ is sufficiently large (if
the function $h$ is nonconvex), we show the existence of a step $d =
O(|x-\bar x|)$ in the proximal linearized subproblem \eqnok{pls} with
corresponding objective value close to the critical value $h(\bar c)$.

When the outer function $h$ is locally Lipschitz (or, in particular,
continuous and convex), this result and its proof simplify
considerably.  First, the transversality condition is automatic.
Second, while the proof of the result appeals to the technical tool we
developed in the previous subsection (Theorem~\ref{th:pert}), this tool
is trivial in the Lipschitz case, as we noted earlier. 
% There are similar simplifications in the case in which $h$ is
% feasible at $c(x)$.
We state the theorem in a form that encompasses both the general case
and the specialization to convex $h$.

\begin{theorem}[proximal step] \label{th:proxstep}
Consider a function $h \colon \R^m \rightarrow \bar\R$ and a map $c \colon
\R^n \rightarrow \R^m$.  Suppose that $c$ is $\cC^2$
around the point $\bar x \in \R^n$, that $h$ is prox-regular at the point
$\bar c = c(\bar x)$, and that the composite function $h \circ c$ is critical
at $\bar x$.  Assume the transversality condition
\begin{equation} \label{cq}
\partial^{\infty} h(\bar c) \cap \mbox{\rm Null}(\nabla c(\bar x)^*) ~=~ \{0\}.
\end{equation}
Then there exist numbers $\bar\mu \ge 0$, $\delta>0$, and $\bar\rho \ge 0$, and a mapping
$d: B_{\delta}(\bar{x}) \times (\bar{\mu},\infty) \to \R^n$ such that
the following properties hold.
\begin{itemize}
\item[(a)] For all points $x \in B_{\delta}(\bar{x})$ and all
  parameter values $\mu > \bar{\mu}$, the step $d(x,\mu)$ is a local
  minimizer of the proximal linearized subproblem (\ref{pls}) with 
\[
h \big( c(x) + \nabla c(x) d(x,\mu) \big) + \frac{\mu}{2} | d(x,\mu) |^2 \le 
h \big( c(x) \big),
\]
and moreover $|d(x,\mu)| \le \bar\rho |x-\bar{x}|$.

\item[(b)] Given any sequences $x_r \to \bar{x}$ and $\mu_r >
  \bar{\mu}$, then if either $\mu_r |x_r-\bar{x}|^2 \to 0$ or
  $h\big(c(x_r)\big) \to h(\bar{c})$, we have 
\begin{equation}  \label{eq:happr}
h\big(c(x_r) + \nabla c(x_r) d(x_r,\mu_r)\big) \to h(\bar c).
\end{equation}

\item[(c)] When $h$ is convex and lower semicontinuous, the results of
  parts (a) and (b) hold with $\bar\mu = 0$.
\end{itemize}
\end{theorem}

\begin{proof}
  Without loss of generality, suppose $\bar x = 0$ and
  $\bar{c}=c(0)=0$, and furthermore $h(0)=0$.  By assumption,
$
0 \in \partial(h \circ c)(0) \subset \nabla c(0)^* \partial h(0),
$
using the chain rule \cite[Thm 10.6]{Roc98}, so there exists a vector
$
v \in \partial h(0) \cap \mbox{Null}(\nabla c(0)^*).
$

We first prove part (a). By prox-regularity, there exists a constant
$\rho \ge 0$ such that
\begin{equation} \label{hrho}
h(z) \ge \ip{v}{z} - \frac{\rho}{2}|z|^2
\end{equation}
for all small vectors $z \in \R^m$.  Hence, there exists a constant
$\delta_1 > 0$ such that $\nabla c$ is continuous on $B_{\delta_1}(0)$
and
\[
h_{x,\mu}(d) ~\ge~ \ip{v}{c(x) + \nabla c(x) d} - \frac{\rho}{2}|c(x) + \nabla c(x) d|^2 + \frac{\mu}{2}|d|^2
\]
for all vectors $x,d \in B_{\delta_1}(0)$.  As a consequence, we have that
\[
h_{x,\mu}(d) ~\ge~ \min_{|x|\le \delta_1, \, |d| = \delta_1} \,
\left\{ \ip{v}{c(x) + \nabla c(x) d} - \frac{\rho}{2}|c(x) + \nabla c(x) d|^2
\right\} + \frac{\mu}{2}|d|^2,
\]
and the term in braces is finite by continuity of $c$ and $\nabla c$
on $B_{\delta_1}(0)$. Hence by choosing $\bar{\mu}$ sufficiently large
(certainly greater than $\rho \| \nabla c(0) \|^2$) we can ensure that
$
h_{x,\bar\mu}(d) \ge  1~~\mbox{whenever}~|x| \le \delta_1,~ |d|=\delta_1.
$
Then for $x \in B_{\delta_1}(0)$, $|d| =\delta_1$, and $\mu \ge
\bar{\mu}$, we have
\begin{equation} \label{eq:hmu}
h_{x,\mu}(d) = h_{x,\bar\mu}(d) + \frac12 (\mu-\bar\mu) |d|^2  
\ge 1 + \frac12   (\mu-\bar\mu) \delta_1^2.
\end{equation}

Since $c$ is $\cC^2$ at $0$, there exist constants $\beta > 0$ and
$\delta_2 \in (0,\delta_1)$ such that, for all $x \in B_{\delta_2}(0)$,
the vector
\begin{equation} \label{defz}
z(x) = c(x) - \nabla c(x) x
\end{equation}
satisfies $|z(x)| \le \beta|x|^2$.  Setting $G = \nabla c(x)$,
$\bar{G} = \nabla c(0)$, $\bar{z}=0$, and $z=z(x)$ in
Theorem~\ref{th:pert}, we obtain the following result.  
For some constants $\gamma > 0$ and
$\delta_3 \in (0,\delta_2)$, given any vector $x \in B_{\delta_3}(0)$,
there exists a vector $\hat{d}(x) \in \R^n$ (defined by $\hat{d}(x):=w-x$, in the
notation of the theorem) satisfying
\begin{align*}
|x+\hat{d}(x)|                            & \le \gamma|z(x)|  
\le \gamma \beta|x|^2                                                                                                                                        \\
h\big(c(x) + \nabla c(x) \hat{d}(x) \big) & \le \gamma|z(x)| 
\le \gamma \beta|x|^2.
\end{align*}
We deduce the existence of a constant $\delta_4 \in (0,\delta_3)$ such
that, for all $x \in B_{\delta_4}(0)$, the corresponding $\hat{d}(x)$ satisfies
$
|\hat{d}(x)| \le |x| + \gamma \beta|x|^2 ~<~ \delta_1, 
$
and
\begin{align*}
h_{x,\mu}(\hat{d}(x))                     & = h(c(x)+\nabla c(x)\hat{d}(x)) + \frac{\mu}{2} |\hat{d}(x)|^2                                                   \\
                                          & \le \gamma \beta|x|^2 + \frac{\bar\mu}{2} \big(|x| + \gamma \beta|x|^2\big)^2 + \frac12 (\mu-\bar\mu) \delta_1^2 \\
                                          & < 1 + \frac12 (\mu-\bar\mu) \delta_1^2.
\end{align*}
The lower semicontinuous function $h_{x,\mu}$ must have a minimizer
(which we denote $d(x,\mu)$) over the compact set
$B_{\delta_1}(0)$. Since $d=0$ is feasible for $B_{\delta_1}(0)$, we
must have $h_{x,\mu} \big( d(x,\mu) \big) \le h_{x,\mu}(0) = h \big(
c(x) \big)$. Moreover, the inequality above implies that the
corresponding minimum value is majorized by
$h_{x,\mu}\big(\hat{d}(x)\big)$, and thus is strictly less than $1 +
(1/2) (\mu-\bar\mu) \delta_1^2$.  But inequality (\ref{eq:hmu})
implies that this minimizer must lie in the {\em interior} of the ball
$B_{\delta_1}(0)$; in particular, it must be an unconstrained local
minimizer of $h_{x,\mu}$. By setting $\delta=\delta_4$, we complete
the proof of the first part of (a).  Notice further that for $x
\in B_{\delta_4}(0)$, we have
\begin{align} \label{hlower} 
                                          & h\big(c(x) + \nabla c(x)d(x,\mu)\big)                                                                            \\
\nonumber
                                          & \qquad\qquad ~ \le ~ h_{x,\mu}\big(d(x,\mu)\big) 
~ \le ~ h_{x,\mu}\big(\hat{d}(x)\big) 
~ \le ~  \gamma \beta|x|^2 + \frac{\mu}{2} \big(|x| + \gamma
  \beta|x|^2\big)^2.
\end{align}

We now prove the remainder of part (a), that is, uniform boundedness
of the ratio $|d(x,\mu)|/|x|$. Suppose there are sequences $x_r \in
B_{\delta}(\bar{x})$ and $\mu_r > \bar{\mu}$ such that $|d_r|/|x_r|
\to \infty$, where we use notation $d_r := d(x_r,\mu_r)$ for brevity.
Since $|d_r| \le \delta_1$ by the arguments above, we must have $x_r
\to 0$.  By the arguments above, for all large $r$ we have the
following inequalities:
\begin{align*}
\gamma \beta & |x_r|^2 + \frac{\mu_r}{2} \big(|x_r| + \gamma \beta|x_r|^2\big)^2
                                                            \\
             & \ge
h_{x_r,\mu_r}(d_r)                                          \\
             & \ge
\ip{v}{c(x_r) + \nabla c(x_r) d_r} - \frac{\rho}{2}|c(x_r) + \nabla c(x_r) d_r|^2 + \frac{\mu_r}{2}|d_r|^2. 
%                                                           \\
%            & =
% \ip{v}{z(x_r) + (\nabla c(x_r) - \nabla c(0))(x_r + d_r)} \\
%            & \qquad\qquad -  
% \frac{\rho}{2}|z(x_r) + \nabla c(x_r)(x_r + d_r)|^2 + \frac{\mu_r}{2}|d_r|^2,
\end{align*}
% using the fact that $\nabla c(0)^* v = 0$.  
Dividing each side by
$(1/2) \mu_r |x_r|^2$ and letting $r \rightarrow \infty$, we recall the inequalities
$
\mu_r > \bar{\mu} > \rho \|\nabla c(0)\|^2 \ge 0
$
and observe
that the left-hand side remains finite, while the right-hand side is
eventually dominated by $(1-\rho \| \nabla c(0) \|^2/\mu_r) |
d_r|^2/|x_r|^2$, which approaches $\infty$, yielding a contradiction.

For part (b), suppose first that $\mu_r |x_r|^2 \to 0$.  By
substituting $(x,\mu) = (x_r,\mu_r)$ into \eqnok{hlower}, we have that
\begin{equation}  \label{limsup}
\limsup \, h\big(c(x_r) + \nabla c(x_r)d_r\big) ~ \le ~ 0.
\end{equation}
From part (a), we have that $|d_r|/|x_r|$ is uniformly bounded, hence
$d_r \to 0$ and thus $c(x_r) + \nabla c(x_r)d_r \to 0$.  Being
prox-regular, $h$ is lower semicontinuous at $0$, so
\[
\liminf \, h\big(c(x_r) + \nabla c(x_r)d_r \big) ~ \ge ~ 0.
\]
Combining these last two inequalities gives
$
h\big(c(x_r) + \nabla c(x_r)d_r \big)  \to 0,
$
as required. 

Now suppose instead that $h\big(c(x_r)\big) \to h(\bar{c}) = 0$. We have from
(\ref{hlower}) that
\[
h\big(c(x_r) + \nabla c(x_r) d_r \big) \le h_{x_r,\mu_r}(d_r) \le
h_{x_r,\mu_r}(0) = h\big(c(x_r)\big).
\]
Taking the lim sup, we again obtain (\ref{limsup}),
and the result follows as before.

For part (c), when $h$ is lower semicontinuous and convex, the
argument simplifies.  We set $\rho=0$ in (\ref{hrho}) and choose the
constant $\delta > 0$ so the map $\nabla c$ is continuous on
$B_\delta(0)$.  For constants $\beta$ and $\gamma$ as before,
Theorem~\ref{th:pert} again guarantees the existence, for all small
points $x$, of a step $\hat{d}(x)$ satisfying
$
h\big(c(x) + \nabla c(x) \hat{d}(x) \big) \le \gamma \beta |x|^2.
$
It follows that the proximal linearized objective $h_{x,\mu}$ is
somewhere finite, so has compact level sets, by coercivity.  Thus it
has a global minimizer $d(x,\mu)$ (unique, by strict convexity), which
must satisfy the inequality
\[
h\big(c(x) + \nabla c(x) d(x,\mu)\big) \le 
h\big(c(x) + \nabla c(x) \hat{d}(x) \big) \le
\gamma \beta |x|^2.
\]
The remainder of the argument proceeds as before.
\end{proof}

We elaborate on Theorem~\ref{th:proxstep}(b) by giving a simple
example of a function prox-regular at $c(\bar{x})$ such that for
sequences $x_r \to \bar{x}$ and $\mu_r \to \infty$ that satisfy
neither $\mu_r |x_r-\bar{x}|^2 \to 0$ nor $h\big(c(x_r)\big) \to
h\big(c(\bar{x})\big)$, there exists a sequence of {\em global}
minimizers $d_r := d(x_r,\mu_r)$ of the subproblem \eqnok{pls} for
which (\ref{eq:happr}) is not satisfied. For a scalar $x$, take
$c(x)=x$ and
\[
h(c) = \begin{cases} -c              & (c \le 0)   \\
1+c                                  & (c >0). \end{cases}
\]
The unique critical point is clearly $\bar{x}=0$ with $c(\bar{x})=0$
and $h\big(c(\bar{x})\big)=0$, and this problem satisfies the assumptions of
the theorem. Consider $x>0$, for which the subproblem \eqnok{pls} is
\[
\min_d \, h_{x,\mu}(d) = h(x+d)+ \frac{\mu}{2} d^2 =
\begin{cases} -x-d+\frac{\mu}{2} d^2 & (x+d \le 0) \\
1+x+d+\frac{\mu}{2} d^2              & (x+d >0).
\end{cases}
\]
When $\mu_r x_r \in (0,1]$, then $d_r=-x_r$ is the only local
  minimizer of $h_{x_r,\mu_r}$.  When $\mu_r x_r>1$, the situation is
  more interesting. The value $d_r=-\mu_r^{-1}$ minimizes the
  ``positive'' branch of $h_{x_r,\mu_r}$, with function value
  $1+x_r-(2\mu_r)^{-1}$, and there is a second local minimizer at
  $d_r=-x_r$, with function value $(\mu_r/2) x_r^2$. (In both cases,
  these minimizers satisfy the estimate $|d_r| = O(|x_r-\bar{x}|)$
  proved in part (a).) Comparison of the function values show that in
  fact the global minimum is achieved at the former point
  ($d_r=-\mu_r^{-1}$) when $x_r > \mu_r^{-1} + \sqrt{2}
  \mu_r^{-1/2}$. If this step is taken, we have $x_r+d_r>0$, so the
  new iterate remains on the upper branch of $h$. For sequences $x_r =
  \mu_r^{-1} + 2 \mu_r^{-1/2}$ and $\mu_r \to \infty$, we thus have
  for the global minimizer $d_r=-\mu_r^{-1}$ of $h_{x_r,\mu_r}$ that
  $h(c(x_r)+\nabla c(x_r)d_r) >1$ for all $r$, while
  $h\big(c(\bar{x})\big)=0$, so that \eqnok{eq:happr} does not hold.
  The alternative sequence of local minimizers $d_r = -x_r$ of does,
  however, satisfy the limit \eqnok{eq:happr}.

\subsection{Restoring Feasibility}

In the algorithmic framework to be discussed below, the basic
iteration starts at a current point $x \in \R^n$ such that the
function $h$ is finite at the vector $c(x)$.  We then solve the
proximal linearized subproblem \eqnok{pls} to obtain the step $d =
d(x,\mu) \in \R^n$.  Under reasonable conditions we have shown that,
for $x$ near the critical point $\bar x$, we have $d = O(|x-\bar x|)$
and furthermore we know that the value of $h$ at the vector $c(x) +
\nabla c(x) d$ is close to the critical value $h\big(c(\bar x)\big)$.

The algorithmic idea is now to update the point $x$ to a new point
$x+d$.  When the function $h$ is Lipschitz, this update is motivated
by the fact that, since the map $c$ is $\cC^2$, we have, uniformly for
$x$ near the critical point $\bar x$,
\[
c(x+d) - (c(x) + \nabla c(x) d) ~=~ O(|d|^2)
\]
and hence
\[
h\big(c(x+d)\big) - h\big((c(x) + \nabla c(x) d)\big) ~=~ O(|d|^2).
\]
However, if $h$ is not Lipschitz, it may not be appropriate to update
$x$ to $x+d$: the value $h\big( c(x+d) \big)$ may even be infinite.

In order to take another step, we need somehow to restore
the point $x+d$ to feasibility, or more generally to find a nearby
point with objective value not much worse than our linearized estimate
$h(c(x) + \nabla c(x) d)$.  Depending on the form of the function $h$,
this may or may not be easy computationally.  However, as we now
discuss, our fundamental transversality condition \eqnok{cq},
guarantees that such a restoration is always possible in theory. In
the next section, we refer to this restoration process as an
``efficient projection.''

\begin{theorem}[linear estimator improvement] \label{th:feasproj}
  Consider a map $c \colon \R^n \rightarrow \R^m$ that is $\cC^2$
  around the point $\bar x \in \R^n$, and a lower semicontinuous
  function $h \colon \R^m \rightarrow \bar\R$ that is finite at the
  vector $\bar c = c(\bar x)$.  Assume that the transversality condition
  \eqnok{cq} holds.
%% \[
%% \partial^{\infty} h(\bar c) \cap \mbox{\rm Null}(\nabla c(\bar x)^*) ~=~ \{0\}.
%% \]
Then there exist constants $\gamma$ and $\delta > 0$ such that, for any
point $x \in B_\delta(\bar x)$ and any step $d \in B_\delta(0) \subset
\R^n$ for which $|h(c(x) + \nabla c(x) d) - h(\bar{c})| < \delta$,
there exists a point $\xnew \in \R^n$
satisfying
\begin{equation} \label{eq:feas}
|\xnew - (x+d)| \le \gamma |d|^2
~~\mbox{and}~~
h\big(c(\xnew)\big) ~\le~  h(c(x) + \nabla c(x) d) +
\gamma |d|^2.
\end{equation}
\end{theorem}

\begin{proof}
Define a $\cC^2$ map $F \colon \R^n \times \R \to \R^m \times \R$ by
$F(x,t) = (c(x),t)$.  Notice that the epigraph $\mbox{epi}\, h$ is a
closed set containing the vector $F\big(\bar c,h(\bar c)\big)$.
Clearly we have
\[
\mbox{Null}\Big(\nabla F\big(\bar x,h(\bar c)\big)^*\Big)
~=~
\mbox{Null}\big(\nabla c(\bar x)^*\big) \times \{0\}.
\]
%% On the other hand, using Rockafellar and
%% Wets~\cite[Theorem~8.9]{Roc98}, we have
Recalling the relationship \eqnok{eq:horepi} between
$\partial^{\infty} h$ and $\mbox{epi}\, h$ at $\bar{c}$, we have
\[
(y,0) \in N_{\mbox{\scriptsize epi}\, h}\big(\bar c, h(\bar c) \big)
~~\Leftrightarrow~~
y \in \partial^{\infty} h(\bar c).
\]
Hence the transversality condition is equivalent to 
\[
N_{\mbox{\scriptsize epi}\, h}\big(\bar c, h(\bar c) \big) 
\cap 
\mbox{Null}\Big(\nabla F\big(\bar x,h(\bar c)\big)^*\Big)
~=~ \{0\}.
\]

We next apply Theorem \ref{constraint} to deduce the existence of a
constant $\kappa > 0$ such that, for all vectors $(u,t)$ near the
vector $\big(\bar c, h(\bar c)\big)$ we have
\[
\mbox{dist}\big((u,t),F^{-1}(\mbox{epi}\, h)\big) 
~\le~ 
\kappa \, \mbox{dist}(F(u,t), \mbox{epi}\, h).
\]
Thus there exists a constant $\delta > 0$ such that, for any point $x
\in B_\delta(\bar x)$ and any step $d \in \R^n$ satisfying $|d| \le
\delta$ and $|h(c(x)+\nabla c(x) d) - h(\bar{c})| \le \delta$, we have
\begin{eqnarray*}
\lefteqn{
\mbox{dist}\Big(\big(x+d,h(c(x) + \nabla c(x) d)\big),F^{-1}(\mbox{epi}\, h)\Big) 
}                                                                                      \\
 & \hspace{3cm} & \le~ 
\kappa \, \mbox{dist} \big(F\big(x+d,h(c(x) + \nabla c(x) d)\big), \mbox{epi}\, h\big) \\
 &              & =~ 
\kappa \,  
\mbox{dist} \Big( \big(c(x+d),h(c(x) + \nabla c(x) d)\big), \mbox{epi}\, h\Big)        \\
 &              & \le~
\kappa \, |c(x+d) - (c(x) + \nabla c(x) d)|,
\end{eqnarray*}
since
\[
\big(c(x) + \nabla c(x) d, h(c(x) + \nabla c(x) d)\big) ~\in~ \mbox{epi}\, h.
\]
Since the map $c$ is $\cC^2$, by reducing $\delta$ if necessary we can ensure the existence of a constant $\gamma > 0$ such that the right-hand side of the above chain of inequalities is bounded above by $\gamma |d|^2$.

We have therefore shown the existence of a vector
$
(\xnew,t) \in F^{-1}(\mbox{epi}\, h)
$
satisfying the inequalities
$
|\xnew - (x+d)| \le \gamma |d|^2
$
and
$
|t - h(c(x) + \nabla c(x) d)| \le 
\gamma |d|^2.
$
Since $t \ge h\big(c(\xnew)\big)$, the result follows.
\end{proof}

\subsection{Uniqueness of the Proximal Step and Convergence of Multipliers} 
\label{sec:unique}

% To make further progress, we strengthen slightly our assumptions about
% the functions involved in the composite objective $h \circ c$ that we
% seek to minimize, and about the critical point $\bar x$.  We refine
% these assumptions in two ways:
% \begin{itemize}
% \item
% We assume $h$ is ``subdifferentially regular'' at $\bar c = c(\bar x)$;
% \item
% We strengthen the transversality condition \eqnok{cq}.
% \end{itemize}
% We begin by discussing each refinement in turn.
% 
% Subdifferential regularity was a key idea introduced by Clarke in his
% pioneering work on nonsmooth optimization \cite{Cla83}.  For our
% concrete purposes here, it suffices to recognize that subdifferential
% regularity of a function $h : \R^m \to \bar\R$ at a point $\bar c \in
% \R^m$ is implied by prox-regularity there.  Thus, in particular,
% convex functions, $\cC^2$ functions, and lower-$\cC^2$ functions
% (discussed in the introduction) are all subdifferentially regular.
%
% If the function $h$ is Clarke regular at the point $\bar c$, 

Our focus in this subsection is on uniqueness of the local solution of
(\ref{pls}) near $d=0$, uniqueness of the corresponding multiplier
vector, and on showing that the solution $d(x,\mu)$ of (\ref{pls}) has
a strictly lower subproblem objective value than $d=0$. For the
uniqueness results, we strengthen the transversality condition
(\ref{cq}) to a constraint qualification that we now introduce.

Throughout this subsection we assume that the function $h$ is prox-regular at the point 
$\bar c$.
Since prox-regular functions are subdifferentially regular,
the subdifferential $\partial h(\bar c)$ is a closed and convex set in
$\R^m$, and its recession cone is exactly the horizon subdifferential
$\partial^{\infty} h(\bar{c})$ (see \cite[Corollary~8.11]{Roc98}).
Denoting the subspace parallel to the affine span of the
subdifferential by $\mbox{par}\, \partial h(\bar{c})$, we deduce that
$
\partial^{\infty} h(\bar c) \subset \mbox{par}\, \partial h(\bar c).
$
Hence the ``constraint qualification'' that we next consider, namely
\begin{equation} \label{li}
\mbox{par}\, \partial h(\bar c) \cap \mbox{\rm Null}(\nabla c(\bar x)^*) ~=~ \{0\}
\end{equation}
implies the transversality condition (\ref{cq}).  

Condition \eqnok{li} is related to the linear independence constraint
qualification in nonlinear programming.  To illustrate, consider again
the case of Section \ref{poly}, where the function $h$ is finite and
polyhedral:
\[
h(c) = \max_{i \in I} \{ \ip{h_i}{c} + \beta_i \}
\]
for given vectors $h_i \in \R^m$ and scalars $\beta_i$.
Then, as we noted,
$
\partial h (\bar c) = \mbox{conv} \{ h_i : i \in \bar I\},
$
where $\bar I$ is the set of active indices, so
\[
\mbox{par}\, \partial h (\bar c) ~=~ 
\Big\{ \sum_{i \in \bar I} \lambda_i h_i :  \sum_{i \in \bar I} \lambda_i = 0 \Big\}.
\]
Thus condition \eqnok{li} states
\begin{equation} \label{exli}
\sum_{i \in \bar I} \lambda_i 
\left[
\begin{array}{c}
\nabla c(\bar x)^* h_i \\
1
\end{array}
\right]
=
\left[
\begin{array}{c}
0                      \\
0
\end{array}
\right]
~~\Leftrightarrow~~
\sum_{i \in \bar I} \lambda_i h_i = 0.
\end{equation}
By contrast, the linear independence constraint qualification for the
corresponding nonlinear program \eqnok{classical} at the point
$\big(\bar x,-h(\bar c)\big)$ is
\[
\sum_{i \in \bar I} \lambda_i 
\left[
\begin{array}{c}
\nabla c(\bar x)^* h_i \\
1
\end{array}
\right]
=
\left[
\begin{array}{c}
0                      \\
0
\end{array}
\right]
~~\Leftrightarrow~~
\lambda_i = 0 ~~(i \in \bar I),
\]
which is a stronger assumption than condition \eqnok{exli}.

We now prove a straightforward technical result that addresses two
issues: existence and boundedness of multipliers for the proximal
subproblem \eqnok{pls}, and the convergence of these multipliers to a
unique multiplier that satisfies criticality conditions for
\eqnok{hc}, when the constraint qualification \eqnok{li} is satisfied.
The argument is routine but, as usual, it simplifies considerably in
the case of $h$ locally Lipschitz (or in particular convex and
continuous) around the point $\bar c$, since then the horizon
subdifferential $\partial^{\infty} h$ is identically $\{0\}$ near
$\bar{c}$.

\begin{lemma} \label{lem:convergence}
Consider a function $h \colon \R^m \to \bar\R$ and a map $c \colon
\R^n \to \R^m$.  Suppose that $c$ is $\cC^2$ around the point $\bar x
\in \R^n$, that $h$ is prox-regular at the point $\bar c
= c(\bar x)$, and that the composite function $h \circ c$ is critical
at $\bar x$.  

When the transversality condition \eqnok{cq} holds, then for any
sequences $\mu_r>0$ and $x_r \to \bar x$ such that $\mu_r
|x_r-\bar{x}| \to 0$, and any sequence of critical points $d_r \in
\R^n$ for the corresponding proximal linearized subproblems
(\ref{pls})  satisfying the conditions
\[
% 0 \in \partial h_{x_r,\mu}(d_r),~~ 
d_r = O(|x_r - \bar{x}|) ~~\mbox{and}~~ 
h\big(c(x_r) + \nabla c(x_r) d_r\big) \to h(\bar c),
\]
there exists a bounded sequence of vectors $v_r \in
\R^m$ that satisfy
\begin{subequations} \label{eq:hxmopt.r}
\begin{align} 
\label{eq:hxmopt.r.a}
0   & =  \nabla c(x_r)^* v_r + \mu_r d_r, \\
\label{eq:hxmopt.r.b}
v_r & \in \partial h\big(c(x_r) + \nabla c(x_r) d_r\big).
\end{align}
\end{subequations}
When the stronger constraint qualification (\ref{li}) holds, in place
of \eqnok{cq}, the set of multipliers $v \in \R^m$ solving the
criticality condition \eqnok{hc_necc}, namely
\begin{equation} \label{mult}
\partial h(\bar c) \cap \mbox{\rm Null}(\nabla c(\bar x)^*)
\end{equation}
is in fact a singleton $\{\bar v\}$.  Furthermore, any sequence of
multipliers $\{ v_r \}$ satisfying the conditions above converges to
$\bar{v}$.
\end{lemma}
\begin{proof}
We first assume \eqnok{cq}, and claim that 
\begin{equation} \label{lem:conv:1}
\partial^{\infty} h \big(c(x_r) + \nabla c(x_r) d_r\big) 
\cap \mbox{\rm Null}(\nabla c(x_r)^*) ~=~ \{0\}
\end{equation}
for all large $r$.  Indeed, if this property should fail, then for
infinitely many $r$ there would exist a unit vector $v_r$ lying in the
intersection on the left-hand side, and any accumulation point of these unit
vectors must lie in the set
\begin{equation} \label{zero}
\partial^{\infty} h (\bar c) \cap \mbox{\rm Null}(\nabla c(\bar x)^*),
\end{equation}
by outer semicontinuity of the set-valued mapping $\partial^{\infty}
h$ at the point $\bar c$ \cite[Proposition~8.7]{Roc98}, contradicting
the transversality condition (\ref{cq}).  As a consequence, we can
apply the chain rule \cite[Theorem~10.6]{Roc98} 
to deduce the existence of vectors $v_r \in \R^n$
satisfying (\ref{eq:hxmopt.r}).  This sequence must be bounded, since
otherwise, after taking a subsequence, we could suppose $|v_r| \to
\infty$ and then any accumulation point of the unit vectors $|v_r|^{-1} v_r$
would lie in the set \eqnok{zero}, again contradicting the
transversality condition.  The first claim of the theorem is proved.

For the remaining claims, note first that the
% we assume the constraint qualification (\ref{li}) and note as above
% that it implies the transversality condition (\ref{cq}), so the
chain rule implies that the set \eqnok{mult} is nonempty.  The
constraint qualification (\ref{li}) then implies that this set is a
singleton $\{ \bar{v} \}$. Using boundedness of $\{ v_r\}$, and the
fact that $\mu_r d_r \to 0$, we have by taking limits in
(\ref{eq:hxmopt.r}) that any accumulation point of $\{v_r\}$ lies in
(\ref{mult}) (by $h$-attentive outer semicontinuity of $\partial h$ at
$\bar{c}$), and therefore $v_r \to \bar{v}$.
\end{proof}

% \medskip\noindent {\bf Note:} Under the conditions stated on $\{ \mu_r
% \}$ and $\{ x_r \}$, assuming additionally that $\mu_r > \bar{\mu}$,
% there exists a sequence of local minimizers $d_r$ of $h_{x_r,\mu_r}$
% with the properties stated. This claim follows from part (b) of
% Theorem~\ref{th:proxstep} when we set $d_r = d(x_r,\mu_r)$. We show
% now that such minimizers $d_r$ satisfy $h_{x_r,\mu_r}(d_r) <
% h_{x_r,\mu_r}(0)$.

Using Theorem~\ref{th:proxstep}, we show that the local minimizers of
$h_{x_r,\mu_r}$ satisfy the desired properties, and in addition give a
strict improvement over $0$ in the subproblem (\ref{pls}).

\begin{lemma} \label{lem:drbetter} Consider a function $h \colon \R^m
  \to \bar\R$ and a map $c \colon \R^n \to \R^m$.  Suppose that $c$ is
  $\cC^2$ around the point $\bar x \in \R^n$, that $h$ is prox-regular
  at the point $\bar c = c(\bar x)$, that the composite function $h
  \circ c$ is critical at $\bar x$, and that the transversality
  condition \eqnok{cq} holds.  Then there is a constant $\bar\mu \ge
  0$ with the following property. If $\mu_r > \bar\mu$ and $x_r \to
  \bar{x}$ are sequences such that $\mu_r |x_r-\bar{x}| \to 0$, then
  for all $r$ sufficiently large, we have the following.
\begin{itemize}
\item[(a)] There is a local minimizer $d_r$ of $h_{x_r,\mu_r}$ such
  that
\begin{equation} \label{dr.propA}
d_r = O(|x_r-\bar{x}|) \quad \mbox{\rm and} \quad 
h(c(x_r)+\nabla c(x_r)d_r) \to h(\bar{c}).
\end{equation}
\item[(b)] If $0 \notin \partial (h \circ c)(x_r)$ for all $r$, then
  $d_r \neq 0$ and
\begin{equation} \label{eq:hxm.decr}
h_{x_r,\mu_r}(d_r) < h_{x_r,\mu_r}(0)
\end{equation}
for all $r$ sufficiently large.
\end{itemize}
\end{lemma}
\begin{proof}
Part (a) follows from parts (a) and (b) of Theorem~\ref{th:proxstep}
when we choose $\bar\mu$ as in that theorem and set $d_r =
d(x_r,\mu_r)$.

For part (b), we have from \eqnok{dr.propA} and
Lemma~\ref{lem:convergence} that there exists $v_r$ satisfying
\eqnok{eq:hxmopt.r}. If we were to have $d_r=0$, these conditions
would reduce to
$
\nabla c(x_r)^* v_r = 0
$
and
$ 
v_r \in \partial h\big(c(x_r)\big),
$
so that $0 \in \partial (h \circ c)(x_r)$, by subdifferential
regularity of $h$. Hence we must have $d_r \neq 0$.
To prove \eqnok{eq:hxm.decr}, suppose for contradiction that there are
sequences $\mu_r$, $x_r$ with the assumed properties such that this
inequality does not hold for all $r$ sufficiently large. Without
losing generality, we can assume that \eqnok{eq:hxm.decr} fails to
hold for {\em every} $r$. By taking limits in \eqnok{eq:hxmopt.r} and
from boundedness of $\{ v_r \}$, we can assume without loss of
generality that $v_r \to \bar{v}$, for some $\bar{v}$ with $\nabla
c(\bar{x})^* \bar{v}=0$, $\bar{v} \in \partial h(\bar{c})$, where we
have used $h$-attentive outer semicontinuity of $\partial h (\cdot)$
to obtain the latter inclusion. Let $\rho$ be the constant from
Definition~\ref{def:proxreg} associated with $\bar{c}$ and $\bar{v}$,
and choose $\bar{\mu}$ such that $\bar{\mu} > \rho \| \nabla
c(\bar{x}) \|^2$. By prox-regularity, we have
\begin{alignat*}{2}
h\big(c(x_r)\big)                                   & \ge h(c(x_r)+\nabla c(x_r) d_r) + \ip{v_r}{-\nabla c(x_r) d_r}
- \frac{\rho}{2} |\nabla c(x_r) d_r|^2              &       & \\
                                                    & = h(c(x_r)+\nabla c(x_r) d_r)  + \mu_r |d_r|^2 - 
\frac{\rho}{2} |\nabla c(x_r) d_r|^2                & \quad & 
\mbox{\rm by (\ref{eq:hxmopt.r.a})}                           \\
                                                    & \ge h(c(x_r)+\nabla c(x_r) d_r)  + \frac{\mu_r}{2} |d_r|^2 +
\frac{\mu_r -\rho \| \nabla c(x_r) \|^2}{2} |d_r|^2 &       & \\
                                                    & = h_{x_r,\mu_r}(d_r)  +
\frac{\mu_r -\rho \| \nabla c(x_r) \|^2}{2} |d_r|^2 & \quad & 
\mbox{\rm by (\ref{pls})}                                     \\
                                                    & > h_{x_r,\mu_r}(d_r),
\end{alignat*}
where the final inequality holds because of our choice of
$\bar{\mu}$. Since $h_{x_r,\mu_r}(0) = h\big(c(x_r)\big)$, we have a
contradiction, and the proof is complete.
\end{proof}

Returning to the assumptions of Theorem~\ref{th:proxstep}, but now
with the constraint qualification \eqnok{li} replacing the weaker
transversality condition \eqnok{cq}, we can derive local uniqueness
results about critical points for the proximal linearized subproblem.
When the outer function $h$ is convex, uniqueness is obvious, since
then the proximal linearized objective $h_{\mu,x}$ is strictly convex
for any $\mu > 0$.  For lower $\cC^2$ functions, the argument is much
the same: such functions have the form $g - \kappa | \cdot |^2$,
locally, for some continuous convex function $g$,
so again $h_{\mu,x}$ is locally strictly convex for large $\mu$.  For
general prox-regular functions, the argument requires slightly more
care.

\begin{theorem}[unique step] \label{th:uscm} Consider a function $h
  \colon \R^m \rightarrow \bar\R$ and a map $c \colon \R^n \rightarrow
  \R^m$.  Suppose that $c$ is $\cC^2$ around the point $\bar x \in
  \R^n$, that $h$ is prox-regular at the point $\bar c = c(\bar x)$,
  and that the composite function $h \circ c$ is critical at $\bar x$.
  Suppose further that the constraint qualification (\ref{li})
  holds.  Then there exists  $\bar\mu \ge 0$ such that the
  following properties hold.  Given any sequence $\{ \mu_r \}$ with
  $\mu_r > \bar{\mu}$ for all $r$ and any sequence $x_r \to \bar{x}$
  such that $\mu_r |x_r-\bar{x}| \to 0$, there exists a sequence of
  local minimizers $d_r$ of $h_{x_r,\mu_r}$ and a corresponding
  sequence of multipliers $v_r$ with the following properties:
\begin{equation} \label{dr.prop} 
0 \in \partial h_{x_r,\mu_r}(d_r),~~ d_r =
O(|x_r-\bar{x}|), ~~\mbox{and}~~ h\big(c(x_r) + \nabla c(x_r) d_r\big)
\to h(\bar c), 
\end{equation}
as $r \to \infty$, and satisfying \eqnok{eq:hxmopt.r}, with $v_r \to
\bar{v}$, where $\bar{v}$ is the unique vector that solves the
criticality condition \eqnok{hc_necc}.  Moreover, $d_r$ is uniquely
defined for all $r$ sufficiently large.

In the case of a convex, lower semicontinuous function $h:\R^m \to
(-\infty,+\infty]$, the result holds with $\bar\mu = 0$.
\end{theorem}

\begin{proof}
  Existence of sequences $\{d_r\}$ and $\{v_r\}$ with the claimed
  properties follows from Theorem~\ref{th:proxstep} and
  Lemma~\ref{lem:convergence}, where we select $\bar\mu$ in the same
  way as in Theorem~\ref{th:proxstep}. We need only prove the claim
  about uniqueness of the vectors $d_r$, and the final claim about the
  special case of $h$ convex and lower semicontinuous.

%% Throughout the proof, we choose $\hat{\mu} > \bar{\mu}$, where
%% $\bar{\mu}$ is defined in Theorem~\ref{th:proxstep}.

We first show the uniqueness of $d_r$ in the general case. Since the
function $h$ is prox-regular at $c(\bar x)$, its subdifferential
$\partial h$ has a hypomonotone localization around the point
$(c(\bar x),\bar v)$ with constant $\rho > 0$ (see the Appendix).  
%In other words, for the constant $\rho$ in the
%definition of prox-regularity (Definition~\ref{def:proxreg}) and for
%some $\epsilon>0$, the mapping $T\colon \R^m \tto \R^m$ defined by
%\[
%T(y) = 
%\left\{
%\begin{array}{ll}
%\partial h(y) \cap B_{\epsilon}(\bar v) & 
%(y \in B_{\epsilon}\big(c(\bar x)\big),~ |h(y) - h\big(c(\bar x)\big)| \le \epsilon) \\
%\emptyset                               & 
%(\mbox{otherwise})
%\end{array}
%\right.
%\]
%has the property
%\[
%z \in T(y) ~\mbox{and}~ z' \in T(y') ~~\Rightarrow~~ \ip{z'-z}{y'-y} \ge -\rho|y'-y|^2.
%\]
%(See \cite[Example~12.28 and Theorem~13.36]{Roc98}.)  
% Using this tool,
% we prove our uniqueness result by setting $\hat\mu = \max\{ \bar\mu,
% \rho\|\nabla c(\bar x)\|^2 \}$, where $\bar{\mu}$ is the constant from
% Theorem~\ref{th:proxstep}.
If the uniqueness claim does not hold, we have by taking a subsequence
if necessary that there is a sequence $x_r \to \bar x$ and distinct
sequences of $d_r^1 \ne d_r^2$ in $\R^n$ satisfying the conditions
\[
0 \in \partial h_{x_r,\mu_r}(d_r^i),~~ d_r^i = O(|x_r-\bar{x}|) \to 0, 
~~\mbox{and}~~ 
h\big(c(x_r) + \nabla c(x_r) d_r^i\big) \to h\big(c(\bar x)\big),
\]
as $r \to \infty$, for $i=1,2$.  Lemma \ref{lem:convergence} shows
the existence of sequences of vectors $v_r^i \in \R^n$ satisfying
\begin{eqnarray*}
0                                        & =   & \nabla c(x_r)^* v_r^i + \mu_r d_r^i  \\
v_r^i                                    & \in & \partial h\big(c(x_r) + \nabla c(x_r) d_r^i\big),
\end{eqnarray*}
for all large $r$, and furthermore $v_r^i \to \bar v$ for each
$i=1,2$.  Consequently, for all large $r$ we have
\[
v_r^i \in T \big( c(x_r) + \nabla c(x_r) d_r^i \big)  ~~\mbox{for}~i=1,2,
\]
so that
\[
-\mu_r | d_r^1 - d_r^2 |^2 ~=~ 
\ip{v_r^1 - v_r^2}{\nabla c(x_r)(d_r^1 - d_r^2)} ~\ge~ -\rho |\nabla c(x_r)(d_r^1 - d_r^2)|^2.
\]
Since $\bar{\mu} > \rho \| \nabla c(\bar{x}) \|^2$, we
have the contradiction
$
\rho \|\nabla c(x_r)\|^2 \ge \mu_r > \bar{\mu} > \rho \|\nabla c(\bar x)\|^2
$ 
for all large $r$.

For the special case of $h$ convex and lower semicontinuous, we have
from Theorem~\ref{th:proxstep}(c) that unique $d_r$ with the
properties \eqnok{dr.prop} exists, for $\bar{\mu}=0$.
%% Uniqueness of $d_r$ follows from strict convexity of
%% $h_{x_r,\mu_r}$. Validity of the chain rule, which is needed to
%% obtain \eqnok{eq:hxmopt.r}, follows as in the proof of
%% Lemma~\ref{lem:convergence}.
\end{proof}

\subsection{Manifold Identification}
\label{sec:id}

We next work toward the identification result.  Consider a sequence of
points $\{x_r \}$ in $\R^n$ converging to the critical point $\bar{x}$
of the composite function $h \circ c$, and let $\mu_r$ be a sequence
of positive proximality parameters.  Suppose now that the outer
function $h$ is partly smooth at the point $\bar c = c(\bar{x}) \in
\R^m$ relative to some manifold $\cM \subset \R^m$.  Our aim is to
find conditions guaranteeing that the update to the point $c(x_r)$
predicted by minimizing the proximal linearized objective
$h_{x_r,\mu_r}$ lies on $\cM$: in other words,
\[
c(x_r) + \nabla c(x_r) d_r \in \cM ~~\mbox{for all large $r$,}
\]
where $d_r$ is the unique small critical point of
$h_{x_r,\mu_r}$.  We would furthermore like to ensure that the ``efficient projection'' $\xnew$ resulting from this prediction, guaranteed by Theorem \ref{th:feasproj} (linear estimator improvement), satisfies 
$c(\xnew) \in \cM$.

To illustrate, we return to our ongoing example from Section
\ref{poly}, the finite polyhedral function
%% the case in which the outer function $h$ is finite and
%% polyhedral,
%% \[
%% h(c) = \max_{i \in I} \{ \ip{h_i}{c} + \beta_i \},
%% \]
%% for some given vectors $h_i \in \R^m$ and scalars $\beta_i$ (see
(\ref{eq:hpoly}).  If $\bar I$ is the active index set corresponding
to the point $\bar c$, then it is easy to check that $h$ is partly
smooth relative to the manifold
\[
\cM ~=~ 
\big\{ 
c : \ip{h_i}{c} + \beta_i = \ip{h_j}{c} + \beta_j ~\mbox{for all}~i,j \in \bar I 
\big\}.
\]

Our analysis requires one more assumption, in addition to those of
Theorem~\ref{th:uscm}.  The basic criticality condition
\eqnok{hc_necc} requires the existence of a multiplier vector:
\[
\partial h(\bar c) \cap \mbox{\rm Null}(\nabla c(\bar x)^*) ~\ne~ \emptyset.
\]
We now strengthen this assumption slightly, to a ``strict''
criticality condition:
\begin{equation} \label{ri}
\ri\big(\partial h(\bar c)\big) \cap \mbox{\rm Null}(\nabla c(\bar x)^*) ~\neq~ 
\emptyset,
\end{equation}
where $\ri$ denotes the relative interior of a convex set.  The
condition \eqnok{ri} is related to the strict complementarity
assumption in nonlinear programming.  For finite polyhedral $h$
\eqnok{eq:hpoly}, since $\partial h(\bar c) = \mbox{conv}\{ h_i : i
\in \bar I\}$, we have
\[
\ri \big( \partial h(\bar c) \big) ~=~ 
\Big\{ 
\sum_{i \in \bar I} \lambda_i h_i : \sum_{i \in \bar I} \lambda_i = 1,~ \lambda > 0 
\Big\}.
\]
Hence, the strict criticality condition \eqnok{ri} becomes
the existence of a vector $\lambda \in \R^{\bar I}$ satisfying
\begin{equation}
\lambda > 0 
~~~\mbox{and}~~~
\sum_{i \in \bar I} \lambda_i 
\left[
\begin{array}{c}
\nabla c(\bar x)^* h_i \\
1
\end{array}
\right]
=
\left[
\begin{array}{c}
0                      \\
1
\end{array}
\right].
\end{equation}
The only change from the corresponding basic criticality condition \eqnok{first} is that the condition $\lambda \ge 0$ has been strengthened to $\lambda > 0$, corresponding exactly to the extra requirement of strict complementarity in the nonlinear programming formulation \eqnok{nlp}.

Recall that the constraint qualification \eqnok{li} implies the uniqueness of the multiplier vector $\bar v$, by Lemma~\ref{lem:convergence}.  Assuming in addition the strict criticality condition \eqnok{ri}, we then have 
\[
\bar{v} \in \ri\big(\partial h(\bar c)\big) \cap \mbox{\rm Null}(\nabla c(\bar x)^*).  
\]

%We use the following result from Hare and
%Lewis~\cite{HarL04}, establishing a relationship between partial
%smoothness of functions and sets.
%%
%\begin{theorem} (\cite[Theorem~5.1]{HarL04}) \label{th:HL5.1}
%A  function $h \colon \R^m \to \bar\R$ is partly smooth at a point $\bar{c} \in \R^m$
%relative to a manifold $\cM \subset \R^m$ if and only if the restriction $h|_{\cM}$ is $\cC^2$
%around $\bar{c}$ and the epigraph $\mbox{\rm epi} \, h$ is partly smooth at the point
%$\big(\bar{c},h(\bar{c})\big)$ relative to the manifold 
%$\{ \big(c,h(c)\big) : c \in \cM \}$.
%\end{theorem}

We now prove a trivial modification of \cite[Theorem~5.3]{HarL04}.
\begin{theorem} \label{th:HL5.3}
Suppose the function $h\colon \R^m \to \bar\R$ is partly smooth at the point
$\bar{c} \in \R^m$ relative to the manifold $\cM \subset \R^m$, and is prox-regular there.
Consider a subgradient $\bar{v} \in \ri \, \partial h(\bar{c})$.  Suppose the sequence 
$\{ \hat{c}_r \} \subset \R^m$ satisfies $\hat{c}_r \to \bar{c}$ and 
$h(\hat{c}_r) \to h(\bar{c})$.  Then
$\hat{c}_r \in \cM$ for all large $r$ if and only if
$\mbox{\rm dist} \big(\bar{v}, \partial h(\hat{c}_r)\big) \to 0$.
\end{theorem}

\begin{proof}
The proof proceeds exactly as in \cite[Theorem~5.3]{HarL04}, except
that instead of defining a function $g: \R^m \times \R \to \R$ by
$g(c,r) = r$, we set $g(c,r) = r - c^T \bar{v}$.
\end{proof}

We can now prove our main identification result.

\begin{theorem} \label{th:id1} Consider a function $h \colon \R^m
  \rightarrow \bar\R$, and a map $c \colon \R^n \rightarrow \R^m$ that
  is $\cC^2$ around the point $\bar x \in \R^n$.  Suppose that $h$ is
  prox-regular at the point $\bar c = c(\bar x)$, and partly smooth
  there relative to the manifold $\cM$.  Suppose further that the
  constraint qualification (\ref{li}) and the strict criticality
  condition \eqnok{ri} both hold for the composite function $h \circ
  c$ at $\bar x$.  Then there exist nonnegative constants $\hat \mu$ and
  $\gamma$ with the following property. Given any sequence $\{
  \mu_r \}$ with $\mu_r > \hat{\mu}$ for all $r$, and any sequence
  $x_r \to \bar{x}$ such that $\mu_r |x_r-\bar{x}| \to 0$, the local
  minimizer $d_r$ of $h_{x_r,\mu_r}$ defined in Theorem~\ref{th:uscm}
  satisfies, for all large $r$, the condition
\begin{equation} \label{identification}
c(x_r) + \nabla c(x_r) d_r \in \cM, 
\end{equation}
and also the inequalities
\begin{equation} \label{correction}
| \xsupnew_r - (x_r + d_r) | \le \gamma |d_r|^2
~~\mbox{and}~~
h \big( c(\xsupnew_r) \big) \le
h \big( c(x_r) + \nabla c(x_r) d_r \big) + \gamma|d_r|^2,
\end{equation}
hold for some point $\xsupnew_r$ with 
$c(\xsupnew_r) \in \cM$.

In the special case when $h:\R^m \to (-\infty,+\infty]$ is convex and
  lower semicontinuous, the result holds with $\hat\mu = 0$.
\end{theorem}

\begin{proof}
Theorem~\ref{th:uscm} implies $d_r \to 0$, so 
$
\hat{c}_r = c(x_r) + \nabla c(x_r) d_r \to \bar c.
$
The theorem also shows $h(\hat{c}_r) \to h(\bar c)$, and that there exist multiplier vectors $v_r \in \partial h(\hat{c}_r)$ satisfying
$
v_r \to \bar v \in \ri\, \partial h(\bar c).
$
Since
$
\mbox{dist} \big(\bar{v}, \partial h(\hat{c}_r)\big) \le 
| \bar{v} - v_r | \to 0,
$
we can apply Theorem~\ref{th:HL5.3} to obtain property (\ref{identification}).

Let us now define a function $h_{\cM} \colon \R^m \to \bar \R$, agreeing with $h$ on the manifold $\cM$ and taking the value $+\infty$ elsewhere.  By partial smoothness, $h_{\cM}$ is the sum of a smooth function and the indicator function of $\cM$, and hence 
$\partial^{\infty} h_{\cM}(\bar c) = N_{\cM}(\bar c)$.  Partial smoothness also implies
$\mbox{par}\big( \partial h(\bar c) \big) = N_{\cM}(\bar c)$.  We can therefore rewrite the
constraint qualification (\ref{li}) in the form
$
\partial^{\infty} h_{\cM}(\bar c) \cap \mbox{Null} \big( \nabla c(\bar x)^* \big) = \{0\}.
$
This condition allows us to apply Theorem \ref{th:feasproj} (linear
estimator improvement), with the function $h_{\cM}$ replacing the
function $h$, to deduce the existence of the point $\xsupnew_r$, as
required.
\end{proof}

\section{A Proximal Descent Algorithm}
\label{sec:proxdesc}

We now describe Algorithm ProxDescent, a simple first-order algorithm
that manipulates the proximality parameter $\mu$ in (\ref{pls}) to
achieve a ``sufficient decrease'' in $h$ at each iteration. (This
algorithm is shown in the figure below as
Algorithm~\ref{alg:proxdescent}.) We follow our description with
results concerning the global convergence behavior of this method and
its ability to identify the manifold $\cM$ discussed in
Section~\ref{sec:id}.

\begin{algorithm}[ht!]
\caption{ProxDescent}
\label{alg:proxdescent}
\begin{algorithmic}
\STATE  Define constants $\tau>1$, $\sigma \in (0,1)$, 
and $\mumin > 0$; 
\STATE
 Choose $x_0 \in \R^n$, $\mu_0 \ge \mumin$; 
\STATE Set $\mu \leftarrow \mu_0$; 
\FOR{$k=0,1,2,\dotsc$}
\STATE Set accept $\leftarrow$ false; 
\WHILE{not accept}
\IF{$d=0$ is a local minimizer of \eqref{pls}}
\STATE Terminate with $\bar{x}=x_k$;
\ENDIF
\STATE Find a local minimizer $d$ of \eqnok{pls}  with $x=x_k$, that is
\[
\min_d \,  h_{x_k,\mu}(d) := 
h\big(c(x_k)+ \nabla c(x_k) d\big) + \frac{\mu}{2} |d|^2,
\]
 such that $h_{x_k,\mu}(d)  < h_{x_k,\mu}(0)$;
\IF{no such $d$ exists}
\STATE $\mu \leftarrow \tau \mu$;
\ELSE
\STATE Derive $x^+$ from $x_k+d$ (by an efficient projection and/or 
 other enhancements); 
\IF{$h\big(c(x_k)\big) - h\big(c(x^+)\big) \ge \sigma  
\left[ h\big(c(x_k)\big) - h(c(x_k)+\nabla c(x_k) d) \right]$ \\
\quad\quad\quad and $|x^+-(x_k+d)| \le \half |d|$}
\STATE $x_{k+1} \leftarrow x^+$;
\STATE $d_k \leftarrow d$; 
\STATE $\mu_k \leftarrow \mu$; 
\STATE $\mu \leftarrow \max(\mumin, \mu/\tau)$; 
\STATE accept $\leftarrow$ true;                              \\
\ELSE
\STATE $\mu  \leftarrow \tau \mu$;
\ENDIF
\ENDIF
\ENDWHILE
\ENDFOR
\end{algorithmic}
\end{algorithm}

A few remarks about Algorithm ProxDescent are in order. First, we are
not specific about the derivation of $x^+$ from $x_k+d$, but we assume
that the ``efficient projection'' technique that is the basis of
Theorem~\ref{th:feasproj} is used when possible.
Lemma~\ref{lem:drbetter} indicates that for $\mu$ sufficiently large
and $x$ near a critical point $\bar{x}$ of $h \circ c$, it is indeed
possible to find a local solution $d$ of (\ref{pls}) which satisfies
$h_{x,\mu}(d)<h_{x,\mu}(0)$ as required by the algorithm, and which
also satisfies the conditions of Theorem~\ref{th:feasproj}.
Lemma~\ref{lem:accept} below shows further that the new point $x^+$
satisfies the acceptance tests in the algorithm. However,
Lemma~\ref{lem:accept} is more general in that it also gives
conditions for acceptance of the step when $x_k$ is {\em not} in a
neighborhood of a critical point of $h \circ c$.

Second, we note that the framework allows $x^+$ to be improved
further. For example, we could use higher-order derivatives of $c$ to
take a further step along the manifold of $h$ identified by the
subproblem (\ref{pls}) (analogous to an ``EQP step'' in nonlinear
programming) and reset $x^+$ accordingly if this step produces a
reduction in $h \circ c$. We discuss this point further at the end of
the section.

%% The main result in this section --- Theorem~\ref{th:glob} ---
%% specifies conditions under which Algorithm ProxDescent does not
%% have nonstationary accumulation points.

We start our convergence analysis with a technical result showing that
in the neighborhood of a non-critical point $\bar{x}$, and for bounded
$\mu$, the steps $d$ do not become too short.

\begin{lemma} \label{lem:noncrit1} Consider a function $h \colon \R^m
  \rightarrow \bar\R$ and a map $c \colon \R^n \rightarrow \R^m$.  Let
  $\bar{x}$ be such that: $c$ is $\cC^1$ near $\bar{x}$; $h$ is finite
  at the point $\bar{c}=c(\bar{x})$ and subdifferentially regular
  there; the transversality condition (\ref{cq}) holds; but the
  criticality condition \eqnok{hc_necc} is {\em not} satisfied.  Then
  there exists a quantity $\epsilon > 0$ such that for any sequence
  $x_r \to \bar{x}$ with $h\big(c(x_r)\big) \to h(\bar{c})$, and any
  sequence $\{\mu_r\}$ with $\mu_r \ge \mumin$, any sequence of
  critical points $d_r$ of $h_{x_r,\mu_r}$ satisfying $h_{x_r ,
    \mu_r}(d_r) \le h_{x_r ,\mu_r}(0)$ must also satisfy $\liminf_r
  \mu_r |d_r| \ge \epsilon$.
\end{lemma}
\begin{proof}
If the result were not true, there would exist sequences $x_r$,
$\mu_r$, and $d_r$ as above except that $\mu_r d_r \to 0$. We would
then have $0 \le | d_r | \le \mu_r | d_r| /\mumin \to 0$. Noting that
$h(c(x_r)+\nabla c(x_r) d_r) \to h(\bar{c})$ (using lower
semicontinuity and the fact that the left-hand side is dominated by
$h\big( c(x_r) \big)$, which converges to $h(\bar c)$), we have that
\[
\partial^{\infty} h \big(c(x_r) + \nabla c(x_r) d_r\big) 
\cap \mbox{\rm Null}(\nabla c(x_r)^*) ~=~ \{0\},
\]
for all $r$ sufficiently large. (If this were not true, we could use
an $h$-attentive outer semicontinuity argument based on
\cite[Proposition~8.7]{Roc98} to deduce that $\partial^{\infty} h (\bar c)
\cap \mbox{\rm Null}(\nabla c(\bar x)^*)$ contains a nonzero vector,
thus violating the transversality condition \eqnok{cq}.)  Hence, we
can apply the chain rule and deduce that there are multiplier vectors
$v_r$ such that \eqnok{eq:hxmopt.r} is satisfied, that is,
\begin{align*}
0   & =  \nabla c(x_r)^* v_r + \mu_r d_r, \\
v_r & \in \partial h\big(c(x_r) + \nabla c(x_r) d_r\big),
\end{align*}
for all sufficiently large $r$.  If the sequence $\{ v_r\}$ is
unbounded, we can assume without loss of generality that $|v_r| \to
\infty$. Any accumulation point of the sequence $v_r/|v_r|$ would be a
unit vector in the set $\partial^{\infty} h(\bar{c}) \cap \mbox{\rm
  Null}(\nabla c(\bar x)^*)$, contradicting \eqnok{cq}. Hence, the
sequence $\{ v_r \}$ is bounded, so by taking limits in the conditions
above and using $\mu_r d_r \to 0$ and outer semicontinuity of
$\partial h(c)$ at $\bar{c}$, we can identify a vector $\bar{v}$ such
that $\bar{v} \in \partial h(\bar{c}) \cap \mbox{\rm Null}(\nabla
c(\bar x)^*)$.  Using the chain rule and subdifferential regularity,
this contradicts non-criticality of $\bar{x}$.
\end{proof}

% Notice the following easy property.
% \begin{lemma}
%   If the function $h \colon \R^m \rightarrow \bar\R$ is bounded below,
%   and the map $c \colon \R^n \rightarrow \R^m$ is differentiable, then
%   for any sequence of points $\{x_r\}$ in $\R^n$ with $\big\{h\big(
%   c(x_r) \big)\big\}$ uniformly bounded, any sequence of parameter
%   values $\mu_r \to +\infty$, and any sequence of steps $\{d_r\}$ in
%   $\R^m$ satisfying the descent condition $h_{x_r,\mu_r}(d_r) \le
%   h_{x_r,\mu_r}(0)$ for each $r=1,2,3,\ldots$, we must have $d_r \to
%   0$.
% \end{lemma}
% \begin{proof}
% By assumption, we have
% \[
% \frac{|d_r|^2}{2} ~\le~ 
% \frac{h\big(c(x_r)\big) - h\big(c(x_r)+ \nabla c(x_r) d_r\big)}{\mu_r}. 
% \]
% The right-hand side converges to zero, since the first term in the numerator is uniformly bounded above and the second is uniformly bounded below.
% \end{proof}

The next result makes use of the efficient projection mechanism of
Theorem~\ref{th:feasproj}. When the conditions of this theorem are
satisfied, we show that Algorithm ProxDescent can perform the
projection to obtain the point $x_k^+$ in such a way that
\eqnok{eq:feas} is satisfied. 
% We thus have the following result.
\begin{lemma} \label{lem:accept} 
Consider a function $h \colon \R^m \rightarrow \bar\R$ and a map $c
\colon \R^n \rightarrow \R^m$ that is ${\cal C}^2$ around a point
$\bar x \in \R^n$. Assume that $h$ lower semicontinuous and finite at
$\bar{c}=c(\bar{x})$ and that transversality condition \eqnok{cq}
holds at $\bar{x}$ and $\bar{c}$.  Then there exist constants
$\tilde{\mu}>0$ and $\tilde{\delta} > 0$ with the following property:
For any $x \in B_{\tilde{\delta}}(\bar{x})$, $d \in
B_{\tilde{\delta}}(0)$, and $\mu \ge \tilde{\mu}$ such that
\begin{equation} \label{beer.1}
h_{x,\mu}(d) \le h_{x,\mu}(0), \qquad 
|h(c(x)+\nabla c(x) d) - h\big(c(\bar{x})\big) | < \tilde{\delta}, 
\end{equation}
there is a point $x^+ \in \R^n$ such that 
\begin{subequations} \label{beer.2and3}
\begin{align}
\label{beer.2}
h \big( c(x)\big) - h\big(c(x^+) \big) & \ge 
\sigma \big[h\big(c(x)\big) - h\big(c(x) + \nabla c(x) d\big)\big],                                 \\
\label{beer.3}
|x^+ - (x+d)|                          & \le \half |d|.
\end{align}
\end{subequations}
\end{lemma}
\begin{proof}
  Define $\delta$ and $\gamma$ as in 
  Theorem~\ref{th:feasproj} and set
$
\tilde{\delta}=\min\big(\delta,1/(2\gamma)\big)
$.
 By
  applying Theorem~\ref{th:feasproj}, we obtain a point $x^+$
%  (denoted by $\xnew$  in the earlier result) 
for which $|x^+-(x+d)| \le \gamma |d|^2 \le \half |d|$ (thus
satisfying (\ref{beer.3})) and $h\big(c(x^+)\big) \le h(c(x)+\nabla
c(x)d) + \gamma |d|^2$. Also note that because of $h_{x,\mu}(d) \le
h_{x,\mu}(0)$, we have
\[
h(c(x)+\nabla c(x)d) + \frac{\mu}{2}|d|^2 \le h\big(c(x)\big)
\]
and hence
\[
|d|^2 \le \frac{2}{\mu} \left[ h\big(c(x)\big) - h(c(x)+\nabla c(x)d) \right].
\]
We therefore have
\begin{align*}
h\big(c(x)\big) - h\big(c(x^+)\big) 
                                       & \ge h\big(c(x)\big) - h(c(x)+\nabla c(x)d)  - \gamma |d|^2 \\
                                       & \ge \left[ h\big(c(x)\big) - h(c(x)+\nabla c(x)d) \right] 
\left( 1-\frac{2 \gamma}{\mu} \right).
\end{align*}
By choosing $\tilde{\mu}$ large enough that $1-2\gamma/\tilde{\mu} >
\sigma$, we obtain (\ref{beer.2}).
\end{proof}

%
% \begin{lemma} \label{lem:accept}
% Consider constants $\sigma \in (0,1)$ and $\gamma > 0$, a function 
% $h \colon \R^m \rightarrow \bar\R$, a differentiable map 
% $c \colon \R^n \rightarrow \R^m$, and a point $\bar x \in \R^n$ such that 
% $h \big( c(\bar x) \big)$ is finite.  Then there exists a constant $\tilde\mu > 0$ such that the conditions
% \begin{equation} \label{simple} 
% \mu \ge \tilde\mu, ~~
% h_{x,\mu}(d) \le h_{x,\mu}(0), ~~
% h \big( c(x^+) \big) \le h(c(x) + \nabla c(x)d) + \gamma|d|^2,
% \end{equation}
% implies the condition
% \begin{equation} \label{rum.1}
% h \big( c(x)\big) - h\big(c(x^+) \big) ~\ge~ 
% \sigma \big[h\big(c(x)\big) - h\big(c(x) + \nabla c(x) d\big)\big]. 
% \end{equation}
% \end{lemma}
% \begin{proof}
% We have the two inequalities
% \begin{eqnarray*}
% h \big( c(x) \big) - h\big(c(x^+) \big) 
% & \ge & 
% h \big( c(x) \big) - \big( h(c(x) + \nabla c(x)d) + \gamma|d|^2 \big) \\
% h \big( c(x) \big) - h(c(x) + \nabla c(x)d)
% & \ge & \frac{\mu}{2}|d|^2.
% \end{eqnarray*}
% The first inequality implies that if we subtract the right-hand side from the left-hand side in inequality (\ref{rum.1}), we obtain at least
% \[
% (1-\sigma)\big[ h\big(c(x)\big) - h\big(c(x) + \nabla c(x) d\big) \big] - \gamma|d|^2
% ~\ge~
% \big( (1-\sigma)\frac{\mu}{2} - \gamma \big) |d|^2,
% \]
% and this quantity is nonnegative for all sufficiently large $\mu$.
% \end{proof}

We also need the following elementary lemma.

\begin{lemma} \label{lem:tech1}
For any constants $\tau>1$ and $\rho>0$ and any positive integer
$t$, we have
\[
\min
\Big\{
\sum_{i=1}^t \alpha_i^2 \tau^i : \sum_{i=1}^t \alpha_i \ge \rho,~ \alpha \in \R^t_+
\Big\}
~~>~~
\rho^2 (\tau-1).
\]
\end{lemma}
\begin{proof}
  By scaling, we can suppose $\rho = 1$.  Clearly the optimal solution
  of this problem must lie on the hyperplane $H = \{ \alpha : \sum_i
  \alpha_i = 1 \}$.  The objective function is convex, and its
  gradient at the point $\bar \alpha \in H$ defined by
\[
\bar \alpha_i = \frac{\tau^{1-i} - \tau^{-i}}{1 - \tau^{-t}} > 0
\]
is easily checked to be orthogonal to $H$.  Hence $\bar \alpha$ is
optimal, and the corresponding optimal value is easily checked to be
strictly larger than $\tau - 1$.
\end{proof}

For the main convergence result, we make the additional assumption
that $h$ can be bounded below, globally, by a (concave) quadratic
function, that is,
\begin{equation} \label{eq:proxbounded}
h(c) \ge h_0 - q_0 |c|^2 \quad \mbox{for all $c \in \R^m$},
\end{equation}
for some scalars $h_0$ and $q_0\ge 0$.  Such functions are called {\em
  prox-bounded} \cite{Roc98}.  This assumption holds for all $h$
considered in the examples of Section~\ref{sec:examples}. The other
assumptions made on $h$, $c$, and $\bar{x}$ in the theorem below allow
us to apply both Lemmas~\ref{lem:noncrit1} and \ref{lem:accept}.

\begin{theorem}[global convergence] \label{th:glob} 
Consider a function $h \colon \R^m \rightarrow \bar\R$ and a map $c
\colon \R^n \rightarrow \R^m$. Suppose that the sequence $\big(x_k,
h\big(c(x_k)\big)\big)$ generated by Algorithm ProxDescent has an
accumulation point at $\big(\bar{x},h(\bar{c})\big)$, where $\bar{c}
:= c(\bar{x})$. Suppose that $c$ is $\cC^2$ near $\bar{x}$, that $h$
is subdifferentially regular (thus lower semicontinuous) at
$\bar{c}$ and is prox-bounded, and that the transversality condition
\eqref{cq} holds at $\bar{x}$. Then the criticality condition
\eqref{hc_necc} is satisfied at $\bar{x}$.
%%  that is $\cC^2$ near a point $\bar{x}
%% \in \R^n$. Denoting $\bar{c}=c(\bar{x})$, suppose that $h$ is
%% subdifferentially regular at $\bar{c}$ and is prox-bounded, and that
%% the transversality condition (\ref{cq}) holds. If
%% $\big(\bar{x},h(\bar{c})\big)$ is an accumulation point for the
%% sequence $\big(x_k, h\big(c(x_k)\big)\big)$ generated by Algorithm
%% ProxDescent, then the criticality condition \eqnok{hc_necc} is
%% satisfied at this point.
\end{theorem}
\begin{proof} Suppose for contradiction that $\big(\bar{x}, h(\bar
  c)\big)$ is an accumulation point but is not critical.  Since the
  sequence $\{ h\big(c(x_r)\big) \}$ generated by the algorithm is
  monotonically decreasing, we have $h\big(c(x_r)\big) \downarrow
  h(\bar{c})$. By the acceptance test in the algorithm and the
  definition of $h_{x,\mu}$ in (\ref{pls}), we have that
\begin{align}
\nonumber
h\big(c(x_{r+1})\big)                     & \le h\big(c(x_{r})\big) - \sigma 
\big[ h\big(c(x_{r})\big) - h\big(c(x_{r}) + \nabla c(x_{r}) d_{r}\big) \big] \\
\label{eq:glob1}
                                          & \le h\big(c(x_{r})\big) - \sigma \frac{\mu_{r}}{2} | d_{r} |^2.
\end{align}
We thus have
\begin{align*}
h\big(c(x_0)\big) - h\big(c(\bar{x})\big) & \ge
\sum_{r=0}^{\infty} h\big(c(x_r)\big) - h\big(c(x_{r+1})\big)                 \\
                                          & \ge \frac{\sigma}{2} \sum_{r=1}^{\infty}  \mu_{r} |d_{r}|^2 
~\ge~ \frac{\sigma}{2} \mumin \sum_{r=1}^{\infty} |d_r|^2,
\end{align*}
which implies that $d_r \to 0$.  Further, we have that
\begin{align}
\nonumber
|h(c(x_r)+                                & \nabla c(x_r) d_r) - h(\bar{c}) | \\
\nonumber
                                          & \le \left[ h\big(c(x_r)\big) - h(c(x_r)+\nabla c(x_r)d_r) \right] +
\left[ h\big(c(x_r)\big) - h(\bar{c}) \right]                                 \\
\label{eq:hnear2}
                                          & \le \sigma^{-1} \left[ h\big(c(x_r)\big) - h\big(c(x_{r+1})\big) \right] + 
\left[ h\big(c(x_r)\big) - h(\bar{c}) \right] \to 0.
\end{align}

Because $\bar{x}$ is an accumulation point, we can define a
subsequence of indices $r_j$, $j=0,1,2,\dotsc$ such that $\lim_{j \to
  \infty} x_{r_j} = \bar{x}$. The corresponding sequence of
regularization parameters $\mu_{r_j}$ must be unbounded, since
Lemma~\ref{lem:noncrit1} indicates that $\liminf_j \mu_{r_j} | d_{r_j}
| \ge \epsilon >0$. Defining $\tilde{\mu}$ and $\tilde{\delta}$ as in
Lemma~\ref{lem:accept}, we can assume without loss of generality that
$\mu_{r_j}>\tau \tilde{\mu}$ and $\mu_{r_{j+1}} > \mu_{r_j}$ for all
$j$. Moreover, since $x_{r_j} \to \bar{x}$ and $d_r \to 0$, and
using (\ref{eq:hnear2}), we can assume that
\begin{subequations} \label{lem:accept:conds}
% \begin{align}
% \label{lem:accept:conds.x}
% x_{r_j}                            & \in B_{\tilde{\delta}/2}(\bar{x}), \quad \mbox{\rm for $j=0,1,2,\dotsc$,}      \\
% \label{lem:accept:conds.d}
% d_{r}                              & \in B_{\tilde{\delta}}(0), \quad \mbox{\rm for all $r>r_0$,}                   \\
% \label{lem:accept:conds.h}
% | h(c(x_{r})                       & + \nabla c(x_{r}) d_{r}) - h(\bar{c}) | \le
% \tilde{\delta}, \quad \mbox{\rm for all $r>r_0$.}
% \end{align}
\begin{alignat}{2}
\label{lem:accept:conds.x}
x_{r_j}                              & \in B_{\tilde{\delta}/2}(\bar{x}), &  & \quad \mbox{\rm for $j=0,1,2,\dotsc$,} \\
\label{lem:accept:conds.d}
d_{r}                                & \in B_{\tilde{\delta}}(0),         &  & \quad \mbox{\rm for all $r>r_0$,}      \\
\label{lem:accept:conds.h}
\big| h(c(x_{r}) + \nabla c(x_{r}) d_{r} & ) - h(\bar{c}) \big| \le
\tilde{\delta},                      &                                    & \quad \mbox{\rm for all $r>r_0$.}
\end{alignat}
\end{subequations}

Suppose first that there are infinitely many $r_j$, $j=0,1,2,\dotsc$,
such that $\mu$ is increased in an inner iteration of iteration
$r_j$. Without loss of generality, we can assume that this behavior
happens for {\em all} $r_j$, $j=0,1,2,\dotsc$. We consider reasons why
the previously tried value $\mu_{r_j} /\tau$ would have been rejected.
The first possible reason for rejection is that \eqref{pls} does not
have a local minimizer for $x=x_{r_j}$ and
$\mu=\mu_{r_j}/\tau$. Because of \eqref{eq:proxbounded}, we have
\begin{align*}
h \big( c(x_{r_j}) + & \nabla c(x_{r_j}) d \big) + \frac{\mu_{r_j}}{2 \tau} |d|^2 \\
& \ge h_0 - q_0  \big| c(x_{r_j}) + \nabla c(x_{r_j}) d \big|^2 + 
\frac{\mu_{r_j}}{2 \tau} |d|^2 \\
& \ge \left[ h_0 - 2q_0 |c(x_{r_j})|^2 \right] +
\left[ \frac{\mu_{r_j}}{2 \tau} - 2q_0 | \nabla c(x_{r_j})|^2 \right] |d|^2 \\
& \ge \bar{h}_0 + \left( \frac{\mu_{r_j}}{2 \tau} - \bar{q}_0 \right) |d|^2,
\end{align*}
where the last inequality follows from \eqref{lem:accept:conds.x} and
smoothness of $c$. We conclude that $h_{x_{r_j},\mu_{r_j}/\tau}$ has
bounded level sets for all $\mu_{r_j}$ sufficiently large, thus by
lower semicontinuity of $h$, it attains a minimizer
\cite[Theorem~1.9]{Roc98}. Note that $d=0$ is not the minimizer
(otherwise, ProxDescent would have terminated), so at least one of the
local minimizers that exists has $h_{x_{r_j},\mu_{r_j}/\tau}(d) <
h_{x_{r_j},\mu_{r_j}/\tau} (0)$,

We can thus assume that a local minimizer $\hat{d}_{r_j}$ is found at
$x=x_{r_j}$ and $\mu=\mu_{r_j}/\tau$. If $\lim_j \hat{d}_{r_j} =0$,
we have from $h(c(x_{r_j})+\nabla c(x_{r_j}) \hat{d}_{r_j}) \le
h(c(x_{r_j})) \downarrow h(c(\bar{x}))$, the fact that $x_{r_j} \to
\bar{x}$, and lower semicontinuity of $h$ that $h(c(x_{r_j})+\nabla
c(x_{r_j}) \hat{d}_{r_j}) \to h(c(\bar{x}))$. Thus, all conditions of
Lemma~\ref{lem:accept} are satisfied, by $x=x_{r_j}$,
$d=\hat{d}_{r_j}$, $\mu=\mu_{r_j}/\tau$, so it follows from this lemma
that a step would have been taken and 
$\mu$ would {\em not} have been increased above $\mu_{r_j}/\tau$,
a contradiction.  Hence, we must have $\hat{d}_{r_j} \nrightarrow
0$. We can therefore identify a constant $\hat\epsilon>0$ and assume
without loss that $| \hat{d}_{r_j} | \ge \hat\epsilon$ for all
$j=0,1,2,\dotsc$. Since $h_{x_{r_j}, \mu_{r_j}/\tau} (\hat{d}_{r_j}) <
h_{x_{r_j}, \mu_{r_j}/\tau} (0)$, we have
\begin{equation} \label{eq:qbb}
h(c(x_{r_j})+\nabla c(x_{r_j}) \hat{d}_{r_j})  < h(c(x_{r_j})) - \frac{\mu_{r_j}}{2\tau} | \hat{d}_{r_j} |^2.
\end{equation}
Since $\mu_{r_j} \uparrow \infty$, this inequality contradicts
prox-boundedness for all $j$ sufficiently large.  To see this, we have
from \eqref{eq:proxbounded} and $|\hat{d}_{r_j}| \ge \hat\epsilon>0$
(for large $j$) that
\begin{align*}
h(c(x_{r_j}) + \nabla c(x_{r_j}) \hat{d}_{r_j}) & \ge
h_0 - q_0  \left| c(x_{r_j}) + \nabla c(x_{r_j}) \hat{d}_{r_j} \right|^2 \\
& \ge \left[ h_0 - 2q_0 | c(x_{r_j})|^2 \right] -
\left[ 2q_0 | \nabla c(x_{r_j}) |^2 \right] | \hat{d}_{r_j}|^2 \\
& > h(c(x_{r_j})) - \frac{\mu_{r_j}}{2\tau} |\hat{d}_{r_j}|^2 \quad \mbox{for all $j$ sufficiently large} \\
& > h(c(x_{r_j})+\nabla c(x_{r_j}) \hat{d}_{r_j}),
\end{align*}
giving a contradiction.  We conclude that the case of $\hat{d}_{r_j}
\nrightarrow 0$ also cannot hold, so there are {\em not} infinitely
many $r_j$, $j=0,1,2,\dotsc$ such that $\mu$ is increased during an
internal iteration of major iteration $r_j$. In fact, we can claim, with no loss of generality,
 that $\mu$ is not increased in iteration $r_j$, so that the first
value of $\mu$ tried in each iteration $r_j$, $j=0,1,2,\dotsc$ is
accepted.

Let $k_j$, $j=1,2,\dotsc$ denote the latest iteration prior to
iteration $r_j$ for which $\mu$ is increased in an internal iteration.
Since $\mu_{r_j} > \mu_{r_{j-1}}$, the index $k_j$ is well defined,
and from the discussion above, we have $r_{j-1} < k_j < r_j$.  Since
no increases are performed internally during iterations $k_j+1,
\dotsc, r_j$, the value $\mu_k$ at each of these steps is the first
value tried, that is,
\begin{equation} \label{eq:rum.5}
\tau \tilde{\mu} < \mu_{r_j} = \tau^{-1} \mu_{r_j-1}
 = \tau^{-2} \mu_{r_j-2} = \dotsc = \tau^{k_j-r_j} \mu_{k_j}.
\end{equation}
 We show first that
\begin{equation} \label{eq:krj}
x_{k_j} - x_{r_j} \to 0. 
\end{equation}
If this limit did not hold, there would be a value $\hat{\delta}>0$
such that $| x_{k_j} - x_{r_j} | \ge \hat{\delta}$ for infinitely
many $j$ --- without loss of generality for {\em all} $j$. From the acceptance criteria in
Algorithm ProxDescent, we have 
\begin{equation} \label{eq:77}
| x_{k+1}-x_k| \le | x_{k+1}-(x_k+d_k) | + |d_k| \le \frac32 |d_k|,
\end{equation}
and so
\[
\hat{\delta} \le | x_{r_j} - x_{k_j} | \le
\sum_{k=k_j}^{r_j-1} |x_{k+1}-x_k| \le \frac32
\sum_{k=k_j}^{r_j-1} |d_k|.
\]
To bound the decrease in objective function over the steps from
$x_{k_j}$ to $x_{r_j}$, we have from \eqref{eq:glob1} and
\eqref{eq:rum.5} that
\begin{align*}
h\big(c(x_{k_j})\big) - h\big(c(x_{r_j})\big) 
&= \sum_{k=k_j}^{r_j-1} h\big(c(x_k)\big) - h\big(c(x_{k+1})\big) \\
& \ge \frac{\sigma}{2} \sum_{k=k_j}^{r_j-1} \mu_k |d_k|^2 
~=~ \frac{\sigma}{2} \mu_{r_j} 
\sum_{k=k_j}^{r_j-1}  \tau^{r_j-k} |d_k|^2.
\end{align*}
To obtain a lower bound on the final summation, we apply
Lemma~\ref{lem:tech1} with $\rho=2 \hat{\delta}/3$ (from
\eqref{eq:77}) and $t=r_j-k_j \ge 1$ to obtain
\[
h\big(c(x_{k_j})\big) - h\big(c(x_{r_j})\big) \ge \frac{\sigma}{2}
\mu_{r_j} (2 \hat{\delta}/3)^2 (\tau-1) \ge \frac{2}{9} \tau
\tilde\mu \hat{\delta}^2 (\tau-1) >0,
\]
where we have used $\mu_{r_j} > \tau \tilde\mu$. This inequality
contradicts $h(c(x_r)) \downarrow h(c(\bar{x}))$, so we conclude that
\eqref{eq:krj} holds. It thus follows from the definition of $\{ r_j
\}$ that
\begin{equation} \label{eq:limkj}
\lim_{j \to \infty} \, x_{k_j} = \lim_{j \to \infty} x_{r_j} + \lim_{j
  \to \infty} (x_{k_j} - x_{r_j})  = \bar{x}.
\end{equation}
An identical argument to the one we used to show that $\mu$ cannot be
increased at iteration $r_j$ can now be applied to the sequence $\{
k_j \}$, $j=0,1,2,\dotsc$, to show that the second-to-last value
$\mu_{k_j}/\tau$ tried at iteration $k_j$ would be been accepted for
all $j$ sufficiently large. This contradicts the definition of
$k_j$. Summarizing these arguments, we conclude that the sequence
$\{r_j \}$ does not exist, so the desired contradiction is obtained,
and $\bar{x}$ must be a critical point.
\end{proof}

We note that this global convergence result (stationarity of
accumulation points) is typical of algorithms for nonlinear
programming and composite nonsmooth optimization; see for example
\cite[Theorem~2.1]{FleS89}, \cite[Theorem~3.1]{Yua85}.

%{\em Manifold result.  There are some catches. For a simple result in
%  which we {\em assume} that $x_r \to \bar{x}$ where $\bar{x}$ has
%  nice properties, we need $\mu_r > \hat{\mu}$ in order to apply
%  Theorem~\ref{th:id1}.  However, we'll more likely have $\mu_r$
%  approaching $\mumin$ if everything is going swimmingly in Algorithm
%  ProxDescent with steps being accepted. We can certainly prove
%  something for the case of convex $h$, for which we can choose
%  $\hat{\mu}=0$, so we include a result of this type, at least as a
%  placeholder.}

To illustrate the idea of identification, we state a simple manifold
identification result for the case when the function $h$ is convex and
finite.
\begin{theorem} \label{th:id.conv}
Consider a function $h:\R^m \to \R$, a map $c:\R^n \to \R^m$, and a
point $\bar{x} \in \R^n$ such that $c$ is $C^2$ near $\bar{x}$ and
that the constraint qualification (\ref{li}) and the strict
criticality condition \eqnok{ri} both hold for the composite function
$h \circ c$ at $\bar x$.  Suppose too that $h$ is convex and
continuous on $\mbox{\rm dom} \, h$ near $\bar{c} :=
c(\bar{x})$. Suppose in addition that $h$ is partly smooth at
$\bar{c}$ relative to the manifold $\cM$.  Then if Algorithm
ProxDescent generates a sequence $x_r \to \bar{x}$, we have that
$c(x_r) + \nabla c(x_r) d_r \in \cM$ for all $r$ sufficiently large.
\end{theorem}
\begin{proof}
Note that $h$, $c$, and $\bar{x}$ satisfy the assumptions of
Theorem~\ref{th:id1}, with $\hat{\mu}=0$.  To apply
Theorem~\ref{th:id1} and thus prove the result, we need to show only
that $\mu_r |x_r - \bar{x}| \to 0$. In fact, we show that
$\{\mu_r\}$ is bounded, so that this estimate is satisfied trivially.

Suppose for contradiction that $\{ \mu_r \}$ is unbounded, so without loss of generality we
can choose an infinite subsequence $\{ r_j \}_{j=0,1,2,\dotsc}$ with
the following properties:
\begin{subequations} \label{eq:irj}
\begin{align}
\label{eq:irj.1}
\mu_{r_j} & \uparrow \infty, \\
\label{eq:irj.2}
\lim_{j \to \infty} x_{r_j} &= \bar{x}, \\
\label{eq:irj.3}
\mu & \mbox{\rm \; was increased at an internal iteration of iteration $r_j$.}
\end{align}
\end{subequations}
Similarly to the proof of Theorem~\ref{th:glob}, we consider the
reasons why the value $\mu_{r_j}/\tau$ was rejected as a possible
value for $\mu$ at iteration $r_j$. Let $\hat{d}_{r_j}$ be the value
of $d$ obtained by solving \eqref{pls} with $x=x_{r_j}$ and
$\mu=\mu_{r_j}/\tau$. If $\lim_{j \to \infty}\hat{d}_{r_j}=0$, we have
from \eqref{eq:irj.1} and \eqref{eq:irj.2} and continuity of $h$ that
the conditions of Lemma~\ref{lem:accept} are satisfied by $x=x_{r_j}$,
$d=\hat{d}_{r_j}$, and $\mu=\mu_{r_j}/\tau$, for all $j$ sufficiently
large. This lemma implies that $\mu_{r_j}/\tau$ would have been
accepted at iteration $r_j$, a contradiction. We must therefore have
$\hat{d}_{r_j} \nrightarrow 0$, so may as well assume that we can identify
$\hat\epsilon>0$ such that $| d_{r_j}| \ge \hat\epsilon$ for all $j$
sufficiently large. Since $h_{x_{r_j}, \mu_{r_j}/\tau} (\hat{d}_{r_j})
< h_{x_{r_j}, \mu_{r_j}/\tau} (0)$, inequality \eqref{eq:qbb} holds.
The assumptions on $h$ imply that $h$ is globally bounded below by a
{\em linear} function (the supporting hyperplane at $c(\bar{x})$, for
example), so as in the proof of Theorem~\ref{th:glob}, inequality
\eqref{eq:qbb} also leads to a contradiction.  We conclude that $\{
\mu_r \}$ is bounded, as claimed.
%
%% From the second acceptance condition in ProxDescent and $x_r \to
%% \bar{x}$, we have 
%% \[
%% |d_r | \le | x_{r+1}-(x_r+d_r)| + |x_{r+1}-x_r| \le \frac12 |d_r| +  |x_{r+1}-x_r|,
%% \]
%% and so $|d_r| \le 2 |x_{r+1}-x_r| \to 0$. This fact, together with
%% condinuity of $h$ near $\bar{x}$, indicates that the conditions of
%% Lemma~\ref{lem:accept} are satisfied by $x_r$ and $d_r$, for all $r$
%% sufficiently large. Hence {\em both} step acceptance conditions in
%% ProxDescent will be satisfied whenever $\mu_r \ge \tilde{\mu}>0$,
%% where $\tilde{\mu}$ is defined in Lemma~\ref{lem:accept}. Thus, after
%% a certain iterate, $\mu$ will not be increased at any inner iteration of 
%%
%% Using Lemma~\ref{lem:accept}, we
%% have that the step acceptance condition of Algorithm ProxDescent is
%% satisfied at $x_r$ for all $\mu \ge \tilde{\mu}$. It follows that for
%% all $r$ sufficiently large, we have in fact that $\mu_r \le \tau
%% \tilde{\mu}$, which leads to the desired result.
\end{proof}

To enhance the step $d$ obtained from (\ref{pls}),
% PLS($x,\mu$), 
we might try to incorporate second-order information
inherent in the structure of the subdifferential $\partial h$ at the
new value of $c$ predicted by the linearized subproblem.
Knowledge of the subdifferential 
% $\partial h \big(c_{\mbox{\scriptsize pred}}(x)\big)$ 
$ \partial h\big(c(x) + \nabla c(x) d \big)$ allows us in principle to
compute the tangent space to $\cM$.  We could then try to ``track''
$\cM$ using second-order information, since both the map $c$ and the
restriction of the function $h$ to $\cM$ are $\cC^2$.

\section{Computational Results} \label{sec:computational}

We present results for Algorithm ProxDescent applied to three problems
drawn from the examples of Section~\ref{sec:examples}.  Our results
are far from exhaustive; the wide range of applications of our
framework make a comprehensive study impossible. Moreover, the
algorithmic framework that we present and analyze is of a bare-bones
nature. Significant improvements in efficiency could be gained by
enhancing it in various ways (for example by making the strategy to
increase and decrease $\mu$ more adaptive) and by customizing it to
the various applications. Our goal here is to show that even the basic
ProxDescent algorithm gives good performance on a diverse set of
applications. Two of our applications are regularized linear
least-squares problems, one with a nonconvex regularizer. The other is
a nonsmooth penalty function from a nonlinear programming application
in power systems.

We start with the following $\ell_1$-regularized least-squares
problem:
% \begin{Exa}[$\ell_1$-regularized least squares] \label{ex:lsl1}
\begin{equation} \label{eq:lsl1}
\min_x \, \frac12 \| Ax-b \|_2^2 + \nu \|x\|_1,
\end{equation}
where $\nu>0$ is a regularization parameter. This problem has been
widely studied in recent years in the context of compressed sensing
\cite{Can06a} (where $m<n$) and LASSO \cite{Tib96} (where typically
$m>n$). As mentioned earlier, Algorithm ProxDescent applied to this
problem is closely related to the SpaRSA algorithm for compressed
sensing; we refer to \cite{WriNF08a} for more detailed numerical
testing.

We use ProxDescent to solve a compressed sensing signal recovery
problem in which $51$ components of a $n=4096$-dimensional vector
$\hat{x}$ were chosen to have nonzero values, and $256$ random linear
observations $A\hat{x}$ were made, where each entry of $A$ is drawn
i.i.d. from a normal distribution with mean zero and standard
deviation $1/(2n)$. Random normal noise of mean $0$ and standard
deviation $10^{-4}/(2n)$ is added to each observation, to yield the
vector $b$ in \eqref{eq:lsl1}. The nonzero values of $\hat{x}$ have a
wide range of magnitudes. We choose the regularization parameter $\nu$
to be $.02\| A^T b \|_{\infty}$, which gives good recovery accuracy,
and use $x=0$ as the starting point. For the parameters in
ProxDescent, we used $\tau=1.25$, $\sigma=.01$, and
$\mumin=10^{-4}$. Termination was declared with the relative change in
function value between two successive iterations dropped below
$10^{-4}$. (We note that because of the sufficient decrease condition
in ProxDescent, this quantity dominates a multiple of the first-order
predicted decrease in $h$, which quantity is zero at a stationary
point.)

Results are shown in Figure~\ref{fig:sparsa}. ProxDescent runs for 92
iterations before declaring convergence. The top subfigure illustrates
the solution $\hat{x}$ (with nonzero components shown as vertical
bars) and the recovered solution $x^*$ (indicated by circles), which
has $25$ nonzero components. Note that $x^*$ appears to have captured
all larger-magnitude components of $\hat{x}$ accurately. The middle
figure plots log of objective function value against iteration number,
showing apparent linear convergence. The bottom plot shows the log of
$\mu_k$ plotted against iteration number $k$. This value shows a slow
downward trend and is not constrained by the minimum value $\mumin$.

\begin{figure} 
\begin{center}
\includegraphics[width=1.0\textwidth]{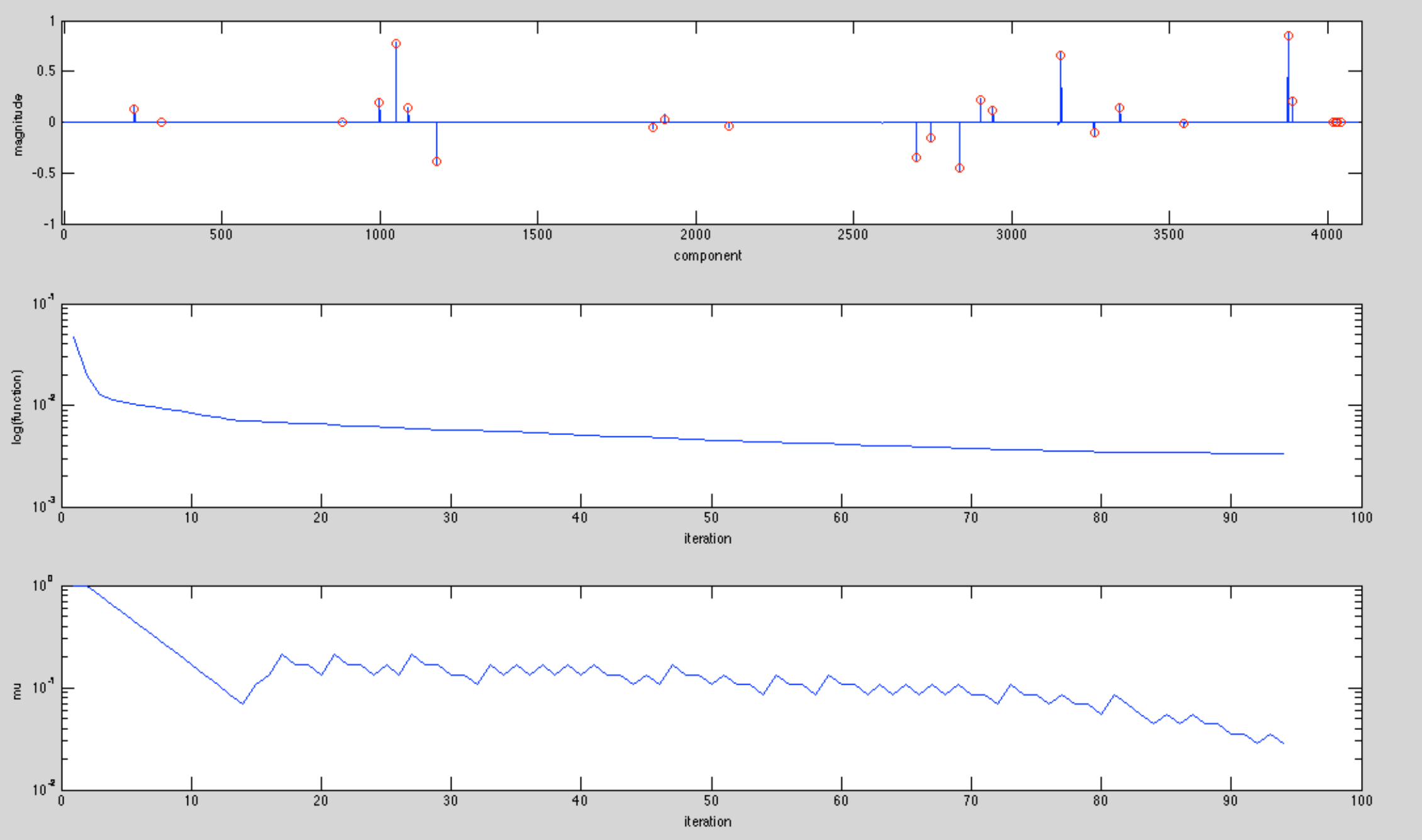}
\end{center}
\caption{Results for formulation (\ref{eq:lsl1}). Top figure shows
  spikes in true signal (bars) and recovered spikes (circles), showing
  accurate recovery of the true signal. Middle figure shows the
  function values at each iteration, while the bottom figure shows the
  values of $\mu_k$ at each iteration $k$.\label{fig:sparsa}}
\end{figure}

% \begin{Exa}[MCP-regularized least squares] \label{ex:mcp}
Consider now the linear least-squares problem with a component-wise
MCP regularized $\phi(\cdot)$ defined in \eqref{eq:mcp}:
\begin{equation} \label{eq:lsmcp}
\min_x \, \frac12 \| Ax-b \|_2^2 + \nu \sum_{i=1}^n \phi(x_i),
\end{equation}
where $\nu>0$ is again the regularization parameter. In replacing the
regularizer $\| \cdot \|_1$ of \eqref{eq:lsl1} with the nonconvex
regularizer of \eqref{eq:lsmcp}, we reduce bias in the solution at the
cost of introducing nonconvexity and thus the possibility of local
minima.
% \end{Exa}

To test ProxDescent on this problem, we used a different random
instance of the same problem as in \eqref{eq:lsl1}, with the same
parameter settings. We define the parameters of the MCP regularizer
\eqref{eq:mcp} to be $\lambda=1$ and $a=\| \hat{x} \|_{\infty}/3$.
These choices ensure that the MCP function has similar slope to $\|
\cdot \|_1$ near zero and that it achieves its maximum value of $a
\lambda^2/2$ for the larger spikes. Results are shown in
Figure~\ref{fig:mcp}, using the same format as Figure~\ref{fig:sparsa}
for the subfigures. Despite the nonconvexity, ProxDescent appears to
have no trouble finding the global minimum of \eqref{eq:lsmcp}, and in
a similar number of iterations as for \eqref{eq:lsl1} ($84$, in this
particular instance).  In fact, a close comparison of the top
subfigures in Figures~\ref{fig:sparsa} and \ref{fig:mcp} indicates
that the recovered spikes in Figures~\ref{fig:sparsa} have slightly
lower magnitudes in general than the true spikes, an effect that is
not present in Figure~\ref{fig:mcp}. This effect is subtle for the
choice of $\nu$ used here (detectable only in high-precision versions
of the plots), but it illustrates nicely the unbiasedness property of
the MCP regularizer. Once again, we see a faint downward trend in the
value of $\mu_k$, and a convergence rate that is clearly linear,
until the solution is identified to high accuracy at about iteration 80.

\begin{figure} 
\begin{center}
\includegraphics[width=1.0\textwidth]{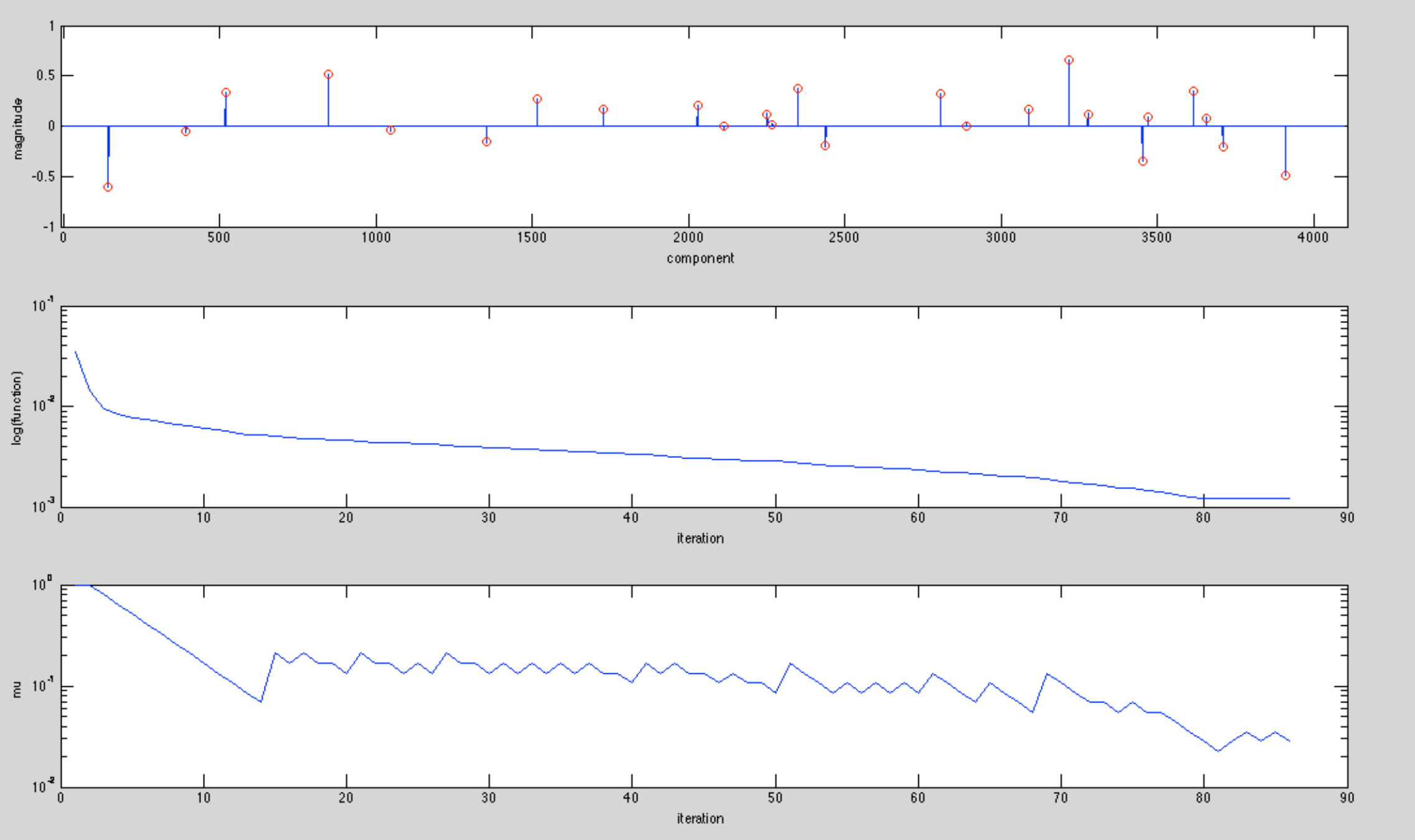}
\end{center}
\caption{Results for formulation (\ref{eq:lsmcp}). Top figure shows
  spikes in true signal (bars) and recovered spikes (circles), showing
  accurate recovery of the true signal. Middle figure shows the
  function values at each iteration, while the bottom figure shows the
  values of $\mu_k$ at each iteration $k$.\label{fig:mcp}}
\end{figure}

% \begin{Exa}[$\ell_1$ penalty function for nonlinear programming] \label{ex:l1p}
Finally, we consider the following nonlinear optimization problem:
\begin{equation} \label{eq:nlp}
\min \, p^Tx \;\; \mbox{subject to} \;\; c(x)=0, \;\; \underline{x}
\le x \le \overline{x},
\end{equation}
where $c:\R^n \to \R^q$ is a smooth nonlinear vector function. A
nonsmooth penalty formulation of this problem, stated in a form
consistent with \eqref{hc}, is as follows:
\begin{equation} \label{eq:nlp.pen}
\min \, p^Tx + \nu \| c(x) \|_1 + I_{[\underline{x},\overline{x}]}(x),
\end{equation}
where the indicator function $I_{[\underline{x},\overline{x}]}(x)$
takes the value $0$ if the bound constraints $ \underline{x} \le x \le
\overline{x}$ are satisfied, and $\infty$ otherwise. This problem was
considered in \cite{KimW14a}, where a sequential $\ell_1$-linear
programming algorithm was proposed to solve it. When ProxDescent is
applied to \eqref{eq:nlp.pen}, the only essential difference between
it and the algorithm of \cite{KimW14a} is that the latter uses an
$\ell_\infty$ (``box-shaped'') trust region on the step $d$ in its
subproblem, in place of the quadratic prox-term $(\mu/2) |d|^2$ of
\eqref{pls}, which is equivalent to an $\ell_2$-norm (circular) trust
region.  (In fact, the code used to obtain numerical results in
\cite{KimW14a} was easily modified to produce the results shown here.)

We use ProxDescent on the framework \eqref{eq:nlp.pen} to solve two
problems from \cite{KimW14a}, arising from the restoration of stable
operation of a power grid following a disruption, such as loss of a
transmission line. In this application, the variables $x$ represent
voltage phasors at each node of the grid and various slacks in the
formulation, while $c(x)$ is derived from the (nonlinear) model of AC
power flow. The bounds on $x$ represent acceptable deviations of
voltage magnitude from $1$, and acceptable values of the amount of
load to be shed from the nodes of the grid. The first problem is of a
type that commonly arises in the power grid application, where the
number of constraints active at the solution of \eqref{eq:nlp} equals
the number of variables, so that methods that use linearization of the
constraints (including ProxDescent and the algorithm of
\cite{KimW14a}) reduce to Newton's method on the system of nonlinear
equations represented by the active constraints, and rapid convergence
is observed once the active set has been determined correctly. In the
second problem, the number of active constraints is fewer than the
number of variables, so rapid convergence cannot be expected from a
first-order method. Here, as in \cite{KimW14a}, convergence is
considerably slower.

% ZZZZZZZZZZZ

For both datasets, we set $\tau=1.5$, $\sigma=10^{-3}$, and $\mumin
=10^{-3}$ in ProxDescent, and terminate when the relative change in
objective falls below $10^{-5}$. Results for the first problem are
shown in Figure~\ref{fig:case57}. This problem has $143$ variables and
$143$ active constraints at the solution. Convergence occurred in $26$
iterations, with a total of $44$ subproblems solved. In
Figure~\ref{fig:case57}, we consider separately the contributions from
the $p^Tx$ term and the penalty term $\| c(x) \|_1$. Both exhibit
steady linear convergence to their optimal values. Two subproblems are
solved on most iterations, because we try to decrease the value of
$\mu$ then increase it again when the smaller value fails to satisfy
the sufficient decrease test.  Note that $\mu_k$ stabilizes at $.058$
on later iterations. Less than one second of execution time was
required on a MacBook Pro (2 GHz Intel i7 with 8GB RAM), using Matlab,
the {\sc MATPOWER} package \cite{ZimMT11} for modeling and solving
power grid problems, and CPLEX. The number of major iterations
required was similar to the S$\ell_1$LP algorithm described in
\cite{KimW14a}. We also coded a version of the algorithm that attempts
to determine the set of active constraints manually once the active
set appears to have settled down, solving a system of nonlinear
equations based on the KKT conditions and making small heuristic
adjustments to the active set in search of a stationary point.  This
version takes 21 iterations, and about half the run time.

\begin{figure} 
\begin{center}
\includegraphics[width=1.0\textwidth]{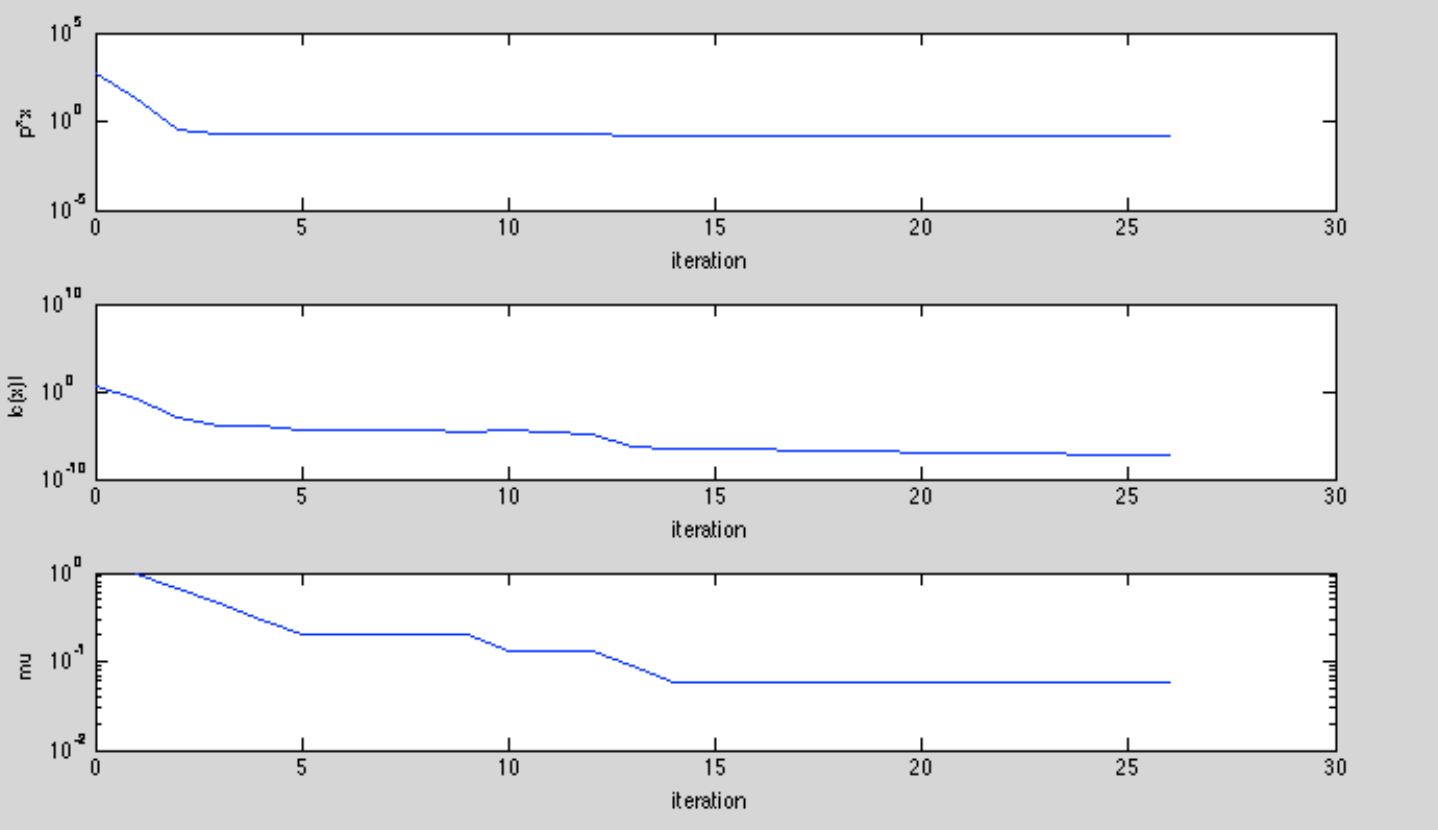}
\end{center}
\caption{Results for formulation (\ref{eq:nlp.pen}), derived from a
  57-bus power grid in which the number of active constraints at the
  solution equals the number of variables. Top figure shows $p^Tx_k$
  plotted against iteration $k$; middle figure show $\|c(x_k)\|_1$;
  bottom figure shows $\mu_k$.\label{fig:case57}}
\end{figure}

The second data set for \eqref{eq:nlp.pen} is derived from a 118-bus
system, and has 262 variables, and 260 constraints active at the
solution. Convergence behavior is quite similar to the first case,
featuring Q-linear convergence at a rate of about $.5$ for the
constraint violation measure $\| c(x_k) \|_1$. $\mu_k$ stabilizes at
the same value $.058$ as for the first data set, and convergence is
declared after $21$ iterations, with $34$ subproblems solved, in about
$.6$ seconds of CPU time. An active-set version of the approach
requires only $11$ iterations and about $.18$ seconds of CPU time.

\begin{figure} 
\begin{center}
\includegraphics[width=1.0\textwidth]{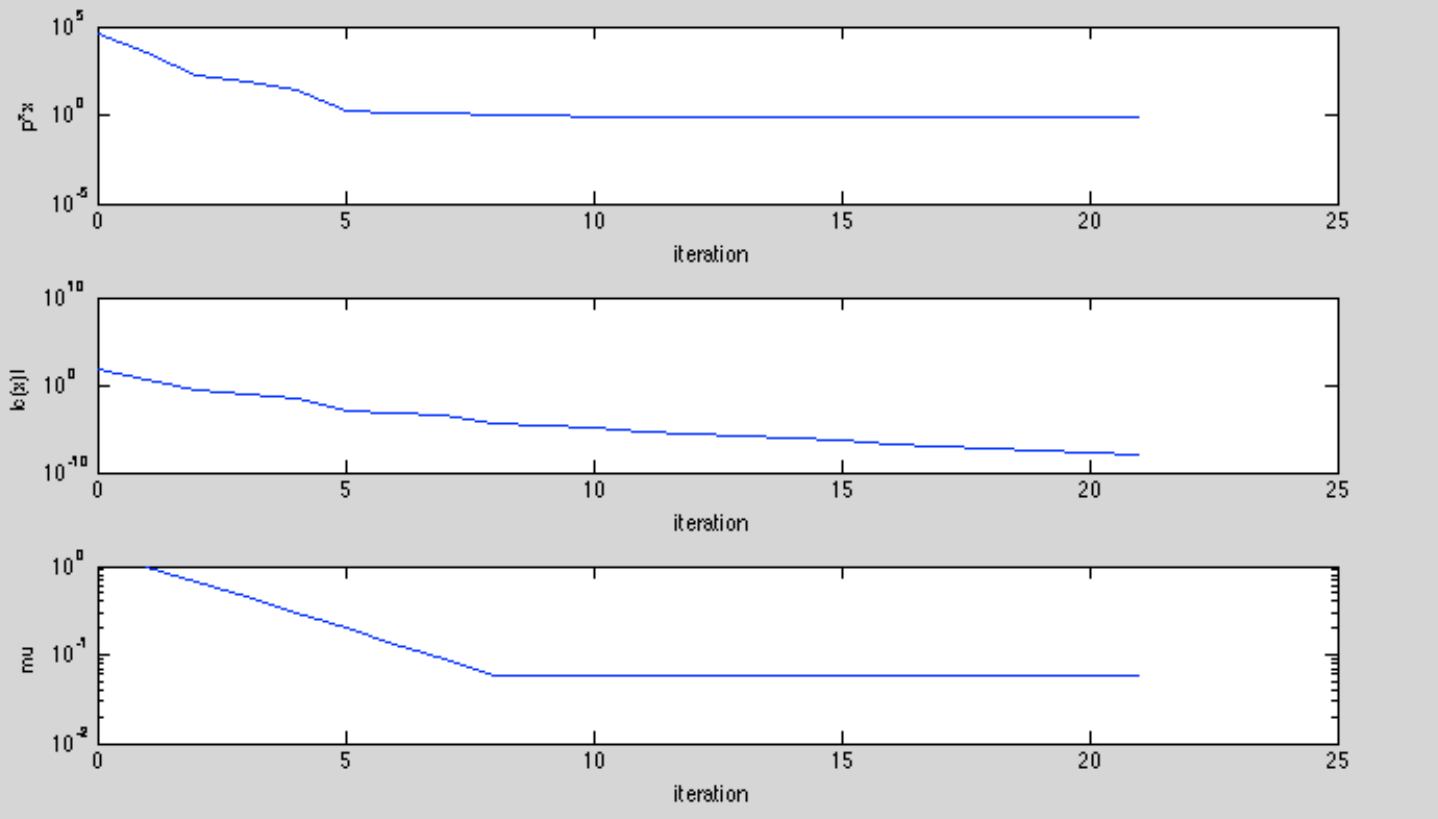}
\end{center}
\caption{Results for formulation (\ref{eq:nlp.pen}), derived from a
  118-bus power grid in which the number of active constraints at the
  solution is smaller than the number of variables. Top figure shows
  $p^Tx_k$ plotted against iteration $k$; middle figure show
  $\|c(x_k)\|_1$; bottom figure shows $\mu_k$.\label{fig:case118}}
\end{figure}

% \end{Exa}

\section*{Acknowledgments}

We acknowledge the support of NSF Grants 0430504 and DMS-0806057.  We
are grateful for the comments of two referees, which were most helpful
in revising earlier versions. We thank Mr. Taedong Kim for obtaining
computational results for the formulation \eqref{eq:nlp.pen}.

\appendix

\section*{Standard theory}
The basic building block
for variational analysis (see Rockafellar and Wets~\cite{Roc98} or
Mordukhovich~\cite{Mor06}) is the {\em normal cone} to a (locally) closed set $S$ at a
point $s \in S$, denoted by $N_S(s)$.  It consists of all {\em normal vectors\/}:  limits of sequences of vectors of the form  $\lambda(u-v)$ for points $u,v \in \R^m$ approaching $s$ such that $v$ is a closest point to $u$ in $S$, and scalars $\lambda > 0$.  On the other hand, {\em tangent vectors} are limits of sequences of vectors 
of the form  $\lambda(u-s)$ for points $u \in S$ approaching $s$ and scalars $\lambda > 0$.  The set $S$ is {\em Clarke regular} at $s$ when the inner product of any normal vector with any tangent vector is always nonpositive.  Closed convex sets and smooth manifolds are everywhere Clarke regular.

The {\em epigraph} of a function $h : \R^m \to \bar\R$ is the set
\[
\mbox{epi} \, h ~=~ \{ (c,r) \in \R^m \times \R : r \ge h(c) \}.
\]
If the value of $h$ is finite at some point $\bar c \in \R^m$, then $h$ is lower semicontinuous nearby if and only if its epigraph is locally closed around the point $\big(\bar c, h(\bar c)\big)$.  Henceforth we focus on that case.

The {\em subdifferential} of $h$ at $\bar c$ is the set
\[
\partial h(\bar c)  ~=~ \big\{ v  \in \R^m \, : \, (v,-1) \in
N_{\mbox{\scriptsize\mbox{epi}} \, h} (\bar{c},h\big(\bar{c})\big) \big\}
\]
and the {\em horizon subdifferential} is
\begin{equation} \label{eq:horepi}
\partial^{\infty} h(\bar c)  ~=~ \big\{ v  \in \R^m  :  (v,0) \in
N_{\mbox{\scriptsize\mbox{epi}} \, h} \big(\bar{c},h(\bar{c})\big) \big\}
\end{equation}
(see \cite[Theorem~8.9]{Roc98}).  The function $h$ is {\em
  subdifferentially regular} at $\bar c$ if its epigraph is Clarke
regular at $\big(\bar c, h(\bar c)\big)$ (as holds in particular if
$h$ is convex lower semicontinuous, or smooth).  Subdifferential regularity implies that
$\partial h(\bar c)$ is a closed and convex set in $\R^m$, and its
recession cone is exactly $\partial^{\infty} h(\bar{c})$ (see
\cite[Corollary~8.11]{Roc98}).  In the case when $h$ is locally
Lipschitz, it is almost everywhere differentiable: $h$ is then
subdifferentially regular at $\bar c$ if and only if its directional
derivative for every direction $d \in \R^m$ equals
\[
\limsup_{c \to \bar c} \ip{\nabla h(c)}{d},
\]
where the $\limsup$ is taken over points $c$ where $h$ is
differentiable.  

Consider a {\em subgradient} $\bar v \in \partial h(\bar c)$, and a {\em localization}
of the subdifferential mapping $\partial h$ around the point
$(\bar c,\bar v)$, by which we mean a set-valued mapping $T\colon \R^m \tto \R^m$ defined by
\[
T(y) = 
\left\{
\begin{array}{ll}
\partial h(y) \cap B_{\epsilon}(\bar v) & 
(|y - \bar c| \le \epsilon,~ |h(y) - h(\bar c)| \le \epsilon) \\
\emptyset &
(\mbox{otherwise})
\end{array}
\right.
\]
for some constant $\epsilon>0$.  The function $h$ is {\em prox-regular at} 
$\bar c$ {\em for} $\bar v$ if some such localization  is {\em hypomonotone\/}:  that is, for some constant $\rho > 0$, 
we have
\[
z \in T(y) ~\mbox{and}~ z' \in T(y') ~~\Rightarrow~~ \ip{z'-z}{y'-y} \ge -\rho|y'-y|^2.
\]
This definition is equivalent to Definition \ref{def:proxreg} (with the same constant $\rho$) \cite[Example~12.28 and Theorem~13.36]{Roc98}.  Prox-regularity at $\bar c$ (for all subgradients $v$) implies subdifferential regularity.

A general class of prox-regular functions common in engineering
applications is ``lower $\cC^2$'' functions
\cite[Definition~10.29]{Roc98}.
A function $h:\R^m \to \R$ is {\em lower} $\cC^2$ around a point $\bar
c \in \R^m$ if $h$ has the local representation
\[
h(c) = \max_{t \in T} f(c,t)  ~~\mbox{for}~ c \in \R^m ~\mbox{near}~ \bar c,
\]
for some function $f:\R^m \times T \to \R$, where the space $T$ is
compact and the quantities $f(c,t)$, $\nabla_c f(c,t)$, and
$\nabla^2_{cc} f(c,t)$ all depend continuously on $(c,t)$. All lower $\cC^2$
functions are prox-regular \cite[Proposition~13.3]{Roc98}. A simple equivalent property, useful in
theory though harder to check in practice, is that $h$ has the form
$g-\kappa |\cdot |^2$ around the point $\bar c$ for some continuous
convex function $g$ and some constant $\kappa$.

The normal cone is crucial to the definition of another central variational-analytic tool.  Given a set-valued mapping $F : \R^p \tto \R^q$ with closed graph, 
\[
\mbox{gph}\,F = \{(u,v) : v \in F(v)\},
\]
at any point $(\bar u,\bar v) \in \mbox{gph}\,F$, the {\em coderivative} 
$D^* F(\bar u | \bar v) : \R^q \tto \R^p$ is defined by
\[
w \in D^* F(\bar u | \bar v)(y) ~\Leftrightarrow~ 
(w,-y) \in N_{\mbox{\scriptsize gph}\,F}(\bar u,\bar v).
\]
The coderivative generalizes the adjoint of the derivative of smooth vector function:  
for smooth $c : \R^n \to \R^m$, the set-valued mapping $x \mapsto F(x) := \{c(x)\}$ has coderivative given by $D^*F(x|c(x))(y) = \{\nabla c(x)^* y\}$ for all $x \in \R^n$ and $y \in \R^m$.  As we see next, coderivative calculations drive two of the arguments in Section \ref{preliminaries}.

\subsubsection*{Proof of Corollary \ref{co:pert}}
Corresponding to any linear map $A \colon \R^p \rightarrow \R^q$, define a
set-valued mapping $F_A \colon \R^p \tto \R^q$ by $F_A(u) = Au-S$.  A
coderivative calculation shows, for vectors $v \in \R^p$,
\[
D^* F_A(0|0)(v) =
\left\{
\begin{array}{ll}
\{A^*v\} & \big(v \in N_S(0)\big) \\
\emptyset & (\mbox{otherwise}).
\end{array}
\right.
\]
Hence, by assumption, the only vector $v \in \R^p$ satisfying $0 \in
D^* F_{\bar A}(0|0)(v)$ is zero, so by \cite[Thm 9.43]{Roc98}, the
mapping $F_{\bar A}$ is metrically regular at zero for zero.  Applying
Theorem \ref{metric} shows that there exist constants $\delta,\gamma
> 0$ such that, if $\|A-\bar A\| < \delta$ and $|v| < \delta$, then we have
\[
\dist\!\big(0,F_A^{-1}(-v)\big) \le \gamma \, \dist\!\big(-v,F_A(0)\big),
\]
or equivalently,
\[
\dist\!\big(0,A^{-1}(S-v)\big) \le \gamma \, \dist (v,S).
\]
Since $0 \in S$, the right-hand side is bounded above by $\gamma |v|$,
so the result follows. ~~~ $\Box$ 

\subsubsection*{Proof of Corollary \ref{constraint}}
  We simply need to check that the set-valued mapping $G \colon \R^p
  \tto \R^q$ defined by $G(z) = F(z) - S$ is metrically regular $\bar
  z$ for zero.  Much the same coderivative calculation as in the proof
  of Corollary~\ref{co:pert} shows, for vectors $v \in \R^p$, the formula
\[
D^* G(\bar z|0)(v) =
\left\{
\begin{array}{ll}
\{\nabla F(\bar z)^*v\} & \big(v \in N_S(\bar z)\big) \\
\emptyset & (\mbox{otherwise}).
\end{array}
\right.
\]
Hence, by assumption, the only vector $v \in \R^p$ satisfying 
$0 \in D^* G(\bar z|0)(v)$ is zero, so metric regularity follows by
\cite[Thm 9.43]{Roc98}. ~~~ $\Box$

\subsubsection*{Alternative proof of Theorem \ref{metric}}
In the text we gave a short ad hoc proof of Theorem \ref{metric}.  Here we present a
more formal approach.  Denote the space of linear maps from $\R^p$ to
$\R^q$ by $L(\R^p,\R^q)$, and define a mapping $g \colon L(\R^p,\R^q)
\times \R^p \to \R^q$ and a parametric mapping $g_H \colon \R^p \to
\R^q$ by $g(H,u)= g_H(u) = Hu$ for maps $H \in L(\R^p,\R^q)$ and
points $u \in \R^p$.  Using the notation of \cite[Section~3]{DmiK08},
the Lipschitz constant $l[g](0;\bar u,0)$, is by definition the
infimum of the constants $\rho$ for which the inequality
\begin{equation} \label{eq:asl1}
d\big(w,g_H(u)\big) \le \rho d\big(u,g_H^{-1}(w)\big)
\end{equation}
holds for all triples $(u,w,H)$ sufficiently near the triple 
$(\bar u, 0, 0)$.  Inequality \eqref{eq:asl1} says simply
\[
|w-Hu| \le \rho |u-z| ~~\mbox{for all $z \in \R^p$ satisfying $Hz=w$},
\]
a property that holds providing $\rho \ge \|H\|$.  We deduce
\begin{equation} \label{eq:asl2}
l[g](0;\bar u,0) = 0.
\end{equation}
We can also consider $F+g$ as a set-valued mapping from
$L(\R^p,\R^q) \times \R^p$ to $\R^q$, defined by $(F+g)(H,u) = F(u)
+ Hu$, and then the parametric mapping $(F+g)_H \colon \R^p \tto
\R^q$ is defined in the obvious way: in other words, $(F+g)_H(u) =
F(u) + Hu$.  According to \cite[Theorem~2]{DmiK08}, equation \eqref{eq:asl2} implies the
following relationship between the ``covering rates'' for $F$ and
$F+g$:
\[
r[F+g](0;\bar u,\bar v) = r[F](\bar u, \bar v).
\]
The reciprocal of the right-hand side is, by definition, the
infimum of the constants $\kappa > 0$ such that inequality
(\ref{unperturbed}) holds for all pairs $(u,v)$ sufficiently near
the pair $(\bar u, \bar v)$.  By metric regularity, this number is
strictly positive.  On the other hand, the reciprocal of the
left-hand side is, by definition, the infimum of the constants
$\gamma > 0$ such that inequality (\ref{perturbed}) holds for all
triples $(u,v,H)$ sufficiently near the pair $(\bar u, \bar v,0)$.

%% \bibliographystyle{siam} 
%% % \bibliography{refsplus}
%% \bibliography{hc}
%% % \bibliography{refs}

\end{document}